\documentclass{amsart}
\usepackage{amsfonts,amsmath,amsthm,amscd,amssymb,latexsym,amsbsy,pb-diagram,caption,url}
\usepackage[space, noadjust]{cite}
\usepackage[mathscr]{eucal}
\usepackage[T2A]{fontenc}
\usepackage[cp1251]{inputenc}
\usepackage[english]{babel}
\usepackage{tikz}
\usetikzlibrary{positioning}
\usepackage{geometry}
\newtheorem{thm}{Theorem}[section]
\newtheorem{cor}[thm]{Corollary}

\newtheorem{prop}[thm]{Proposition}
\newtheorem{claim}[thm]{Claim}
\newtheorem{con}[thm]{Conjecture}
\theoremstyle{definition}
\newtheorem{defn}[thm]{Definition}
\theoremstyle{remark}
\newtheorem{rem}[thm]{Remark}
\newtheorem{exa}[thm]{Example}
\numberwithin{equation}{section}
\DeclareMathOperator{\diam}{diam}
\DeclareMathOperator{\card}{card}
\DeclareMathOperator{\sgn}{sgn}
\DeclareMathOperator{\is}{is}
\DeclareMathOperator{\ac}{ac}
\DeclareMathOperator{\DP}{dp}
\newcommand{\Sp}[1]{\operatorname{Sp}(#1)}
\newcommand{\ve}{\varepsilon}

\newcommand{\NN}{\mathbb{N}}
\newcommand{\RR}{\mathbb{R}}
\newcommand{\mfb}{\mathfrak{B}}
\newcommand{\mfu}{\mathfrak{U}}
\newcommand{\mfo}{\mathfrak{O}}

\usepackage{tkz-euclide}
\begin{document}

\title[Ordinal spaces]{Ordinal spaces}

\author{Karsten Keller}

\address{Institute of Mathematics, University of L\"{u}beck, Ratzeburger Allee 160, 23562 L\"{u}beck,  Germany}

\email{keller@math.uni-luebeck.de}

\author{Evgeniy Petrov}

\address{Institute of Applied Mathematics and Mechanics of NAS of Ukraine,  Dobrovolskogo str. 1, 84100 Slovyansk, Ukraine}

\email{eugeniy.petrov@gmail.com}



\begin{abstract}
Ordinal data analysis is an interesting direction in machine learning. It mainly deals with data for which only the relationships \mbox{`$<$', `$=$', `$>$'} between pairs of points are known. We do an attempt of formalizing structures behind ordinal data analysis by introducing the notion of ordinal spaces on the base of a strict axiomatic approach. For these spaces we study general properties as isomorphism conditions, connections with metric spaces, embeddability in Euclidean spaces, topological properties etc.
\end{abstract}

\maketitle

\section{Introduction}

\noindent
Sometimes it is much easier to gather information about comparisons between objects or persons than to measure exact values related to them, leading to data given on an ordinal scale. Ordinal data usually collected from persons by observation, testing, or questionnaires arise in social, educational, psychological and behavioral sciences, and in governmental and business sectors. Moreover, there are many reasons to investigate given metric data on the pure ordinal level.

Ordinal data analysis has been recently used in machine learning, see, e.g.,~\cite{LK14, LK15,SWS16,CC07,KL171, KL172}. Several real-life motivations for studying machine learning tasks in a setting of ordinal distance information are described in~\cite{KL172}: humans are better and more reliable in assessing dissimilarity on a relative scale than on an absolute one; there are situations where ordinal distance information is readily available, but the underlying dissimilarity function is unknown; there are several applications where actual dissimilarity values between objects can be collected, but it is clear to the practitioner that these values only reflect a rough picture and should be considered informative only on an ordinal scale level.

An example of relative assessments by a person may be the following. Imagine a cyclist with neither a watch nor a speedometer who has ridden all pairwise distances between a set of some localities $A_1,\ldots A_n$. He cannot assign a numerical value to the pair of localities in the form of time or distance. The only thing that the cyclist has is a general impression of the form ``the distance between the localities $A_i$ and $A_j$ is less than, equal to or larger than the distance between localities $A_k$ and $A_l$''. Such assessments generate an ``ordinal structure''.

The aim of our  paper is a formalization of structures of this type and a demonstration that such structures have nontrivial theoretical properties and deserve to be studied. For these purposes we mainly use methods of metric geometry and graph theory.

A straightforward formalization of ``ordinal structure'' leads us to the following definition. For a nonempty set $X$ consider a map

\[
\delta\colon X\times X\times X\times X\longrightarrow \{\mbox{`$<$'},\mbox{`$>$'},\mbox{`$=$'}\}
\]

and the following conditions on $\delta$ for all $x,y,u,v,z,w \in X$:

\begin{itemize}
  \item [(i)] $\delta (x,y,x,y)=\mbox{`$=$'}$;
  \item [(ii)] $\delta (x,y,z,w)=\delta (y,x,z,w)=\delta (x,y,w,z)$;
  \item [(iii)] $\delta (x,y,z,w)=-\delta (z,w,x,y) \mbox{ where } -\mbox{`$<$'}=\mbox{`$>$',\,} -\mbox{`$>$'}=\mbox{`$<$', and } -\mbox{`$=$'}=\mbox{`$=$'}$;
  \item [(iv)] $\delta (x,y,u,v)=\delta (u,v,z,w)=\mbox{`$=$' implies }\delta (x,y,z,w)=\mbox{`$=$'}$;
  \item [(v)] $\delta (x,y,u,v)=\mbox{`$<$' and }\delta (u,v,z,w)\in\{\mbox{`$<$', `$=$'}\}\mbox{ implies }\delta (x,y,z,w)=\mbox{`$<$'}$;
  \item [(vi)] $\delta (x,y,u,v)\in\{\mbox{`$<$', `$=$'}\}\mbox{ and }\delta (u,v,z,w)=\mbox{`$<$' implies }\delta (x,y,z,w)=\mbox{`$<$'}$;
  \item [(vii)] $\delta(x,x,z,w)= \mbox{`$<$' if } z\neq w  \mbox{ and } \delta(x,x,z,w)=\mbox{`$=$' if } z=w$.
\end{itemize}

\begin{defn}\label{d1.1}
$(X,\delta)$ is called an {\em ordinal space} if (i)--(vii) are satisfied.
\end{defn}
Note that the second  equality in (ii) can be omitted since it follows from the first one and condition (iii). In some cases, for convenience, instead of $\delta(x,y,z,w)=\mbox{`$<$',`$=$',`$>$'}$ we will write $\delta(x,y)<, =, >\delta(z,w)$, respectively.

Like the axioms of a metric space all conditions in Definition~\ref{d1.1} express a natural understanding of relations between pairs of points. Condition (i) is similar to the axiom $d(x,x)=0$, (ii) is analogous to the symmetry condition $d(x,y)=d(y,x)$, (iii)  is analogous to relations between real numbers $(a<b)\Leftrightarrow (b>a)$. Conditions (iv), (v), (vi) express transitivity of the relations `$=$' and `$<$' for real numbers and eventually condition (vii) states that the distance from a point to itself is less than a distance between two different points, similarly to the axiom of a metric space $(d(x,y)=0)\Leftrightarrow (x=y)$.

According to axioms (i)--(iv) the relation `$=$' is an equivalence relation on the set of all unordered pairs $\{x,y\}$, $x,y \in X$. Since we know all relations between all pairs of points in $X$, the set of all equivalence classes $(X\times X)/\mbox{`$=$'}$ can be linearly ordered with respect to these relations. This observation leads us to the following equivalent definition of an ordinal space.
\begin{defn}\label{d2.3}
An ordinal space is an ordered triplet $(X,L,\delta)$, where $X$ is a set, $L$ is a linearly ordered set with minimal element $0$ and $\delta$ is a \textbf{surjective} mapping $\delta\colon X\times X \to L$ such that for any $x,y \in X$ the following two axioms hold:
\begin{itemize}
  \item [(i)] $\delta(x,y)=\delta(y,x)$;
  \item [(ii)] $(\delta(x,y)=0)\Leftrightarrow(x=y)$.
\end{itemize}
\end{defn}

Note that according to this definition if $(X,L,\delta)$ is an ordinal space, then $X\neq\varnothing$, as there is $x\in X$ with $\delta(x,x)=0$.

\begin{rem}\label{r1.3}
In Definitions~\ref{d1.1} and~\ref{d2.3} we identify $(X\times X)/\mbox{`$=$'}$ with $L$.
\end{rem}

In what follows we shall use both designations $(X,\delta)$ and $(X,L,\delta)$ for ordinal spaces depending on context.

The introduced concept of ordinal space is closely connected to such concepts as ordinal scaling, multidimensional scaling and ranking, which have a lot of applications, see, e.g.,~\cite{AWC07, BG05, JN11, K64, QY04, RF06, Sh62, Sh66, WJJ13}. The presence of these relationships and, thus, the existence of potential applications, makes the study of ordinal spaces important.

Observe that in some sense an ordinal space is a natural generalization of a semimetric space. Recall that a \emph{semimetric} on a set $X$ is a function $d\colon X\times X\to \mathbb{R}^+$, $\mathbb{R}^+=[0,\infty)$, such that $d(x,y)=d(y,x)$ and $(d(x,y)=0)\Leftrightarrow (x=y)$ for all $x,y \in X$. A pair $(X,d)$, where  $d$  is a semimetric on $X$, is called a \emph{semimetric space} (see, for example, \cite[p.~7]{Bl}). Note  also that a semimetric $d$ is a \emph{metric} if, in addition, the \emph{triangle inequality} $d(x,y)\leqslant d(x,z)+d(z,y)$ holds for all $x, y, z \in X$. A metric is an \emph{ultrametric} if we have the \emph{ultrametric inequality}
$ d(x,y)\leqslant \max \{d(x,z),d(z,y)\}$
instead of the weaker triangle one.

\begin{exa}
Every semimetric space $(X,d)$ may be considered as an ordinal space $(X,\delta)$ if we define $\delta$ for all $x,y,z,w \in X$ in the following way:
\begin{equation}\label{eq1}
\delta(x,y,z,w)=
\begin{cases}
\text{`$<$'}, &\text{if} \ \, d(x,y)<d(z,w);\\
\text{`$=$'}, &\text{if} \ \,d(x,y)=d(z,w);\\
\text{`$>$'}, &\text{if} \ \,d(x,y)>d(z,w).
\end{cases}
\end{equation}
\end{exa}

We shall say that a semimetric space $(X,d)$  has the \emph{ordinal type} $(X, \delta)$ if  equality~(\ref{eq1}) holds for all $x,y,z,w \in X$. In this case we also shall call $(X,d)$ a \emph{realization} of  $(X,\delta)$.

It is easy to see that for every ordinal space $(X,L,\delta)$ with $\card(L)>\mathfrak{c}$ there does not exist a realization which is a semimetric space. Let us call an
ordinal space \emph{semimetric-like} if such a realization exists. It is clear that an ordinal space $(X,L,\delta)$ is semimetric-like if $L\! \setminus \! \{0\}$ is similar to some subset of $\mathbb R^+ \! \setminus \! \{0\}$ (with the usual order) or equivalently to some subset of the real line. (For the concept of similarity of ordered sets, compare the text below Definition \ref{d3.1}.) In~\cite{A77} it was shown that every continuous linearly ordered set $\widetilde{L}$ with no least and greatest element in which some countable subset is dense is order-isomorphic to the real line. Thus $(X,L,\delta)$ is semimetric-like if $L \! \setminus \! \{0\}$ is similar to some subset of such an $\widetilde{L}$. (For the definition of a continuous linearly ordered set, see~\cite{Ha05}.)

The following simple proposition shows that for every semimetric-like ordinal space there exists a realization which is a metric space.

\begin{prop}\label{p1}
Let $(X,L,\delta)$ be a semimetric-like ordinal space. Then there exists a metric $d$ on $X$ such that the metric space  $(X,d)$ has the ordinal type $(X,\delta)$.
\end{prop}

\begin{proof}
Since $(X,L,\delta)$ is semimetric-like and the real line is similar to an open segment $(0,\ve)$, $\ve>0$, there exists
an injective mapping $f\colon L\!\setminus\!\{0\} \to (0,\ve)$ preserving the linear order on $L\!\setminus\!\{0\}$. For every $x,y \in X$ define
\[
d(x,y)=
\begin{cases}
a+f(l), &\text{if} \ \;  x\neq y;\\
0, &\text{if} \ \, x=y;
\end{cases}
\]
where $l\in L\!\setminus\!\{0\}$ is such that $\delta(x,y)=l$ and $a$ is a positive real. It is clear that  the triangle inequality in $(X,d)$ is satisfied for every $x,y,z \in X$ and~(\ref{eq1}) holds for all $x,y,z,w \in X$ with sufficiently small $\ve$ and sufficiently large $a$.
\end{proof}

The concepts of metric and semimetric spaces are natural and well investigated. In some parts of this paper we are interested in studying connections between ordinal spaces and semimetric spaces. It is natural to study these connections on the class of semimetric-like ordinal spaces. Thus, in what follows under {\bf ordinal spaces} $(X,\delta)$ we only understand {\bf semimetric-like ordinal spaces}.

In 1992 Matthews~\cite{Mat} generalized metrics to partial metrics admitting not necessarily zero self-distances. Later, O'Neill~\cite{Neil} in 1995 considered partial metrics by admitting negative distances.
In 1999 Heckmann~\cite{Heck} extended the structure from~\cite{Mat} by omitting the axiom of small self-distances ($d(x,x)\leqslant d(x,y)$ for all $x,y \in (X,d)$). The another well-known generalization is so called quasi-metric spaces which are obtained from metric spaces by omitting the symmetry axiom ($d(x,y)=d(y,x)$ for all $x,y \in (X,d)$).
In this connection, if necessary,  more general classes of spaces than ordinal spaces can be considered in order to have an embedding of such spaces in ordinal structure.

The paper is organized as follows. In Section~\ref{iso} we consider two approaches for defining isomorphism between ordinal spaces based on the different Definitions~\ref{d1.1} and~\ref{d2.3}. This isomorphism is also expressed in terms of weak similarities which form a special class of mappings between semimetric spaces, see Propositions~\ref{p23} and~\ref{p231}.

In Section~\ref{balls} we give an approach for defining the set of balls in ordinal spaces, which is based on the concept of cuts of linearly ordered sets, and classify balls with respect to their ``appearance'' in realizations of ordinal spaces.

Section~\ref{topology} is devoted to topological properties of ordinal spaces. We give some conditions under which an appropriate set of balls in the ordinal space $X$ forms a base of a topology on $X$, see Proposition~\ref{p7.3}. In Proposition~\ref{p314} we show a way of defining a topology on a set $X$ using the structure of the ordinal space $(X,\delta)$. We also describe some connections of ordinal spaces with uniform spaces. It is worth emphasizing here the specificity of the statements in Theorems~\ref{t7.1} and~\ref{p34}. The interesting point of these statements is that it is not necessary to have a triangle inequality for all triplets of points in semimetric spaces in order to have some topological properties for these spaces.

The set of all balls $\mathbf B_X$ of an ordinal space $X$ can be considered as a poset $(\mathbf B_X, \subseteq)$ with the partial order defined by the inclusion $\subseteq$. The Hasse diagram $\mathcal{H}(\mathbf B_X)$ of the poset $(\mathbf B, \subseteq)$ for finite $X$ is a directed graph which for small $X$ gives a clear visual representation of the ball structure of $X$. In Section~\ref{Hasse} we give a criterium for isomorphism of Hasse diagrams $\mathcal{H}(\mathbf B_X)$ and $\mathcal{H}(\mathbf B_Y)$ of finite ordinal spaces $X$ and $Y$, see Theorem~\ref{t25}. It is also established that an isomorphism of the ordinal spaces $X$ and $Y$ implies an isomorphism of the corresponding Hasse diagrams $\mathcal{H}(\mathbf B_X)$ and $\mathcal{H}(\mathbf B_Y)$, but not vise versa, Proposition~\ref{p26}.

In Section~\ref{ballnum} we put forward two hypotheses about maximal and minimal number of balls in finite ordinal spaces.  We conjecture that the sequence of maximal numbers of balls coincides with the sequence A263511 from~\cite{oeis} and the sequence of minimal numbers of balls in $(X,\delta)$ (with some restrictions on $\delta$) coincides with the sequence of triangular numbers.

Section~\ref{rline} is devoted to embeddings of ordinal spaces in the real line. We define a special type of points enumeration in an ordinal space $X$, the so called enumerations with majorization property. Proposition~\ref{p101} claims that for spaces embeddable in the real line such enumeration always exists. The converse assertion is a new interesting conjecture. In Theorem~\ref{t10} we establish that for an ordinal space $X$ with $|X|\leqslant 4$ the existence of an enumeration with majorization property is necessary and sufficient for embeddability in $\RR^1$. Propositions~\ref{p6.8} and~\ref{p210} describe some classes of ordinal spaces embeddable in $\RR^1$.

Section~\ref{hdes} considers embeddings of ordinal spaces in higher dimensional Euclidean spaces. We start this section with a hypothesis that an analog of the famous Menger's theorem holds for ordinal spaces, i.e., that in order to verify that a finite ordinal space is embeddable in $\RR^n$ it is enough to verify the embeddability in $\RR^n$ of each of its $n+3$ point subsets. In Theorem~\ref{t22} we establish that if an ordinal space $X$ is embeddable in $\RR^2$, then the number of its diametrical pairs does not exceed the number of points in $X$ and describe extremal embeddings. In Proposition~\ref{p10} we show that any ordinal space $X$ with $|X|=n+1$ is irreducibly embeddable in $\RR^n$.

In Section~\ref{dist} we propose an analog of the Gromov--Hausdorff distance for the case of ordinal spaces with a fixed finite number of points.

\section{Isomorphism properties of ordinal spaces}\label{iso}
\noindent Based on Definition~\ref{d1.1}, it is natural to define isomorphisms between ordinal spaces as follows.
\begin{defn}\label{d3.1} Ordinal spaces $(X,\delta_X)$  and $(Y,\delta_Y)$ are said to be \emph{isomorphic} if there exists a bijective mapping $\Phi \colon X \to Y$ such that
\begin{equation}\label{e2.1*}
\delta_X(x,y,z,w)=\delta_Y(\Phi(x),\Phi(y),\Phi(z),\Phi(w))
\end{equation}
for all $x,y,z,w \in X$.
\end{defn}

Recall that two linearly ordered sets $(L_1,<_1)$ and $(L_2,<_2)$ are isomorphic or \emph{similar} if there exists a bijection $\Psi\colon L_1\to L_2$ such that for arbitrary $x,y\in L_1$ there holds $(x<y)\Leftrightarrow (\Psi(x)<\Psi(y))$, see, e.g.,~\cite[p.~34--35]{Ha05}. In this case $\Psi$ is called an isomorphism.

Definition~\ref{d2.3} leads us to the following definition of isomorphism.

\begin{defn}\label{d3.1*} Ordinal spaces $(X,L_X,\delta_X)$  and $(Y,L_Y,\delta_Y)$ are said to be  \emph{isomorphic} if $L_X$ and  $L_Y$ are isomorphic as linearly ordered sets and there exists a bijective mapping $\Phi\colon X \to Y$ such that
\begin{equation}\label{e2.2*}
\Psi(\delta_X(x,y))=\delta_Y(\Phi(x),\Phi(y))
\end{equation}
for all $x,y\in X$, where $\Psi\colon L_X\to L_Y$ is an isomorphism of the linearly ordered sets $L_X$ and $L_Y$.
\end{defn}

We have already mentioned the equivalence of isomorphy concepts in Definitions~\ref{d1.1} and~\ref{d2.3}. It is easy to see that isomorphy according to Definition~\ref{d3.1*} implies  isomorphy according to Definition~\ref{d3.1}, but the converse implication we now show is not so evident.

Let $(X,\delta_X)$  and $(Y,\delta_Y)$ be ordinal spaces isomorphic in the sense of Definition~\ref{d3.1} and let $(X,L_X,\delta_X)$  and $(Y,L_Y,\delta_Y)$ be designations of these spaces in the sense of Definition~\ref{d2.3} (see Remark~\ref{r1.3}).  Let us first show that $L_X$ and $L_Y$ are similar. Let $\delta_{1} \in L_X$ and let $\{x,y\}$ be any pair such that $\delta_X(x,y)= \delta_{1}$. Define the mapping $\Psi\colon L_X\to L_Y$ by
\begin{equation}\label{e2.11}
\Psi(\delta_1) = \delta_2
\end{equation}
where $\delta_2=\delta_Y(\Phi(x), \Phi(y))$.

Suppose that $\delta_{X_1}, \delta_{X_2} \in (L_X,<_X)$, $\delta_{X_1}<_X\delta_{X_2}$,  and let $\delta_X(x,y)= \delta_{X_1}$, $\delta_X(z,w)=\delta_{X_2}$. According to the definition of $\Psi$, we get
$$
\Psi(\delta_{X_1}) = \delta_{Y_1}, \, \  \Psi(\delta_{X_2}) = \delta_{Y_2}
$$
with  $\delta_{Y_1}=\delta_Y(\Phi(x),\Phi(y))$ and $\delta_{Y_2}=\delta_Y(\Phi(z),\Phi(w))$. It follows from~(\ref{e2.1*}) that $\delta_Y(\Phi(x),\Phi(y))<\delta_Y(\Phi(z),\Phi(w))$. Hence $\delta_{Y_1}<_Y\delta_{Y_2}$. The monotonicity of $\Psi$ implies its injectivity. The surjectivity easily follows from the surjectivity of $\delta_X$ and $\delta_Y$. Thus the isomorphism of the ordinal spaces $(X,L_X,\delta_X)$  and $(Y,L_Y,\delta_Y)$ implies a similarity of $L_X$ and $L_Y$. It suffices to note only that equality~(\ref{e2.2*}) follows from~(\ref{e2.11}).

The \emph{spectrum} of a semimetric space $(X,d)$ is the set
$\operatorname{Sp}(X)=\{d(x,y)\colon x,y \in  X\}$.

\begin{defn}\label{d4.5}
Let  $(X,d_X)$ and $(Y,d_Y)$ be semimetric spaces. A bijective mapping $\Phi\colon X\to Y$ is called a \emph{weak similarity} if there exists a strictly increasing bijection $f\colon \Sp{X}\to \Sp{Y}$ such that the equality
\begin{equation}\label{e2.1}
f(d_X(x,y))=d_Y(\Phi(x),\Phi(y))
\end{equation}
holds for all $x$, $y\in X$.
\end{defn}

If $\Phi \colon X \to Y$ is a weak similarity, then we say that $X$ and $Y$ are \emph{weakly similar} with the realization $(f, \Phi)$.  The notion of weak similarity was introduced in~\cite{DP2} in a slightly different form.

\begin{prop}\label{p23}
Let $(X,d_X)$ and $(Y,d_Y)$ be semimetric spaces.
The ordinal types $(X,\delta_X)$ and $(Y,\delta_Y)$ are isomorphic if and only if  $(X,d_X)$ and $(Y,d_Y)$ are weakly similar.
\end{prop}
\begin{proof}
Let $(X,d_X)$ and $(Y,d_Y)$ be weakly similar with the realization $(f, \Phi)$. Let us prove that $\Phi$ is an isomorphism of the ordinal types $(X,\delta_X)$ and $(Y, \delta_Y)$. According to~(\ref{e2.1}) we have
\begin{equation}\label{e22}
\begin{split}
&f(d_X(x,y)) = d_Y(\Phi(x), \Phi(y)),\\
&f(d_X(z,w)) = d_Y(\Phi(z), \Phi(w))
\end{split}
\end{equation}
for all $x, y, z, w \in X$.
Since $f$ is strictly increasing, equalities~(\ref{e22}) imply the equivalences
\begin{equation}\label{e23}
\begin{split}
&d_X(x,y)<d_X(z,w) \ \Leftrightarrow \  d_Y(\Phi(x),\Phi(y)) <d_Y(\Phi(z), \Phi(w)),\\
&d_X(x,y)=d_X(z,w) \ \Leftrightarrow \ d_Y(\Phi(x),\Phi(y)) =d_Y(\Phi(z), \Phi(w)),\\
&d_X(x,y)>d_X(z,w) \ \Leftrightarrow \  d_Y(\Phi(x),\Phi(y)) >d_Y(\Phi(z), \Phi(w)).
\end{split}
\end{equation}
Using~(\ref{eq1}) and~(\ref{e2.1*}) we obtain the desired implication.

Conversely, let $(X,d_X)$ and $(Y,d_Y)$ be semimetric spaces and let their ordinal types $(X,\delta_X)$ and $(Y,\delta_Y)$ be isomorphic, with isomorphism $\Phi\colon X\to Y$. Let us prove that $\Phi$ is a weak similarity. According to~(\ref{e2.1*}) and to~(\ref{eq1}) we obtain that relations~(\ref{e23}) hold again.

For all $r\in \Sp{X}$ define
\begin{equation}\label{e4}
  f(r)=d_Y(\Phi(x_0), \Phi(y_0)) \in \Sp{Y}
\end{equation}
where $x_0, y_0 \in X$ is a pair (possibly not unique) such that $d_X(x_0,y_0)=r$. Let us show that $f$ is a strictly increasing bijection. The fact that $f$ is strictly increasing follows form the first equivalence in~(\ref{e23}). Hence $f$ is injective.

To prove the surjectivity suppose that  $s\in \Sp{Y}$. Consequently there exist $u,v\in Y$ such that $d_Y(u,v)=s$. Since $\Phi$ is a bijection, we have $u=\Phi(u_0)$, $v=\Phi(v_0)$ for some $u_0, v_0\in X$. Let $s_0=d_X(u_0,v_0)\in \Sp{X}$. According to~(\ref{e4}) we have $f(s_0)=s$. The surjectivity of $f$ is established. Which completes the proof.
\end{proof}

It is easy to see that the next proposition is an equivalent formulation of Proposition~\ref{p23}, only, of course,
under the restriction given after Proposition \ref{p1} for the whole paper.

\begin{prop}\label{p231}
The ordinal spaces $(X,\delta_X)$ and $(Y,\delta_Y)$ are isomorphic if and only if any their realizations $(X,d_X)$ and $(Y,d_Y)$ are weakly similar.
\end{prop}

\section{What is a ball in an ordinal space?}\label{balls}

\noindent Recall that an open ball of radius $r$ with center $x_0$ in a general semimetric  space $(X,d)$ is defined by
$$
B_r(x_0)=\{x\in X \colon d(x,x_0)<r\}
$$
and a closed ball by
$$
B_r[x_0]=\{x\in X \colon d(x,x_0)\leqslant r\}
$$
where $r>0$ is some positive real and $x_0\in X$. Denote by $\widetilde{\mathbf B}_{X}$ and $\bar{\mathbf B}_{X}$ the sets of all open and closed balls in a semimetric space $(X,d)$, respectively.

The problem in defining balls in ordinal spaces is that we cannot refer to the real numbers as in the semimetric case.
Therefore we want to use the concept of a \emph{cut} in a linearly ordered set $S$ meaning that $S=A\cup B$ and each element in $A$ is less than each element in $B$. We use the designation $(A,B)$ for such a cut and call $A$ the initial segment of $S$ and $B$ the complementary final segment of $S$.

For an ordinal space $(X,L,\delta)$ with $c\in X$ we shall call the set $$\Sp{c}=\{\delta(c,x) \,  | \,  x\in X\}\subseteq L$$ the \emph{spectrum at the point} $c$.
\begin{defn}\label{d2.5}
Let $(X,L,\delta)$ be an ordinal space and
let $(A,B)$ be a cut of the set $\Sp{c}$, $c\in X$, with $A\neq\varnothing$.
The set
$$
B_c=B[c](A,B)=\{x\in X \, | \, \delta(c,x) \in A\}
$$
is called a \emph{ball} at the point $c$ corresponding to the cut $(A,B)$. We denote by $\mathbf B_X$ the set of all balls in the ordinal space $X$.
\end{defn}

\begin{rem}
Note that we do not introduce concepts of open and closed balls for ordinal spaces as they are adopted for metric spaces. The reason is that it is impossible to achieve a compatibility of these concepts in ordinal spaces and their realizations. Below we just classify balls with respect to their appearance in realizations of ordinal spaces. Note also that according to Definition~\ref{d2.5} every one-point set is a ball in $(X,L,\delta)$.
\end{rem}

\begin{claim}
Let $(X,d)$ be any realization of an ordinal space $(X,\delta)$. Then for every open ball $B_r(c)$ or closed one $B_r[c]$ in $(X,d)$ there exists a cut $(A,B)$ of the set $\Sp{c}$ such that $B[c](A,B)=B_r(c)$ or $B[c](A,B)=B_r[c]$, respectively.
\end{claim}
\begin{proof}
In fact, let $c\in X$ and let $r>0$ be a positive real. Let $A=\{\delta(c,x) \, |\, x\in B_r(c)\}$ and $B=\{\delta(c,x) \, |\, x\in X\setminus B_r(c)\}$. Then $\Sp{c}=A\cup B$. As $d(c,c)=0<r$, we have $c\in B_r(c)$ and hence $\delta(c,c)\in A$ implying $A\neq \varnothing$. Now consider $a\in A$ and $b\in B$. Choose $x\in B_r(c)$ and $y\in X \setminus B_r(c)$ with $a=\delta(c,x)$ and $b=\delta(c,y)$. Then $d(x,c)<r$ and $d(y,c)\geqslant r$ implying that $d(x,c)<d(y,c)$ and thus that $d(c,x)<d(c,y)$; it follows that $\delta(c,x,c,y)=\mbox{`$<$'}$ and hence that $\delta(c,x)<\delta(c,y)$, i.e., $a<b$. We conclude that $(A,B)$ is a cut of $\Sp{c}$. For every $x\in X$ we have $\delta(c,x)\in A$ if and only if $x\in B_r(c)$. This shows that $B[c](A,B)=B_r(c)$. The construction of the cut $(A,B)$ in the case of $B_r[c]$ is similar, with $A=\{\delta(c,x) \, | \, x\in B_r[c]\}$.
\end{proof}

Naturally, the following question arises: Can one see from the structure of $(X,\delta)$ whether the set $B[c](A,B)$ is an open or closed ball in an \textbf{arbitrary} realization $(X,d)$? As it will be seen below, in some cases the given set $B[c](A,B)$ is a ball of some type for all realizations $(X,d)$, see also Example~\ref{ex4.1}. In other cases it depends on the structure of the realization $(X,d)$.

Recall some definitions from the theory of ordered sets, see~\cite{Ha05}. A cut $(A,B)$ is called \emph{proper}, if $A$ and $B$ are both non-empty. Let $(S, \leqslant)$ be a linearly ordered set. A proper cut $(A, B)$ in $S$ is said to be of type
\begin{quote}
$(1, 1)$, if $A$ has a last and B has a first element,\\
$(\infty, 1)$, if $A$ has no last and B  has a first element,\\
$(1, \infty)$, if $A$ has a last and B has no first element,\\
$(\infty, \infty)$ if $A$ has no last and B has no first element.\
\end{quote}

Again, let $(X,d)$ be an arbitrary realization of an ordinal space $(X,\delta)$ and let $(A,B)$ be a cut of the set $\Sp{c}$, $c\in X$. The reader can verify that the following holds:

\begin{enumerate}
  \item [(1)] If the linearly ordered set $\Sp{c}\!\setminus\! \{0\}$  has a least element, then $\{c\} \in \widetilde{\mathbf B}_{X}, \bar{\mathbf B}_{X}$.
  \item [(2)] If the cut $(A,B)$ has a type $(1,1)$ or $(\infty,\infty)$, then $B_c\in \widetilde{\mathbf B}_{X}, \bar{\mathbf B}_{X}$.
  \item [(3)] If there exists a point $c \in X$ such that the set $\Sp{c}$ has a last element, then $B[c](\Sp{c},\varnothing) =X \in \widetilde{\mathbf B}_{X}$, $\bar{\mathbf B}_{X}$.
\end{enumerate}

Now consider the other types of cuts of the linearly ordered set $\Sp{c}$.
\begin{itemize}
  \item [(4)] If the linearly ordered set $\Sp{c}\setminus \{0\}$  has no least element, then the point $\{c\}$ in some realization $(X,d)$ can be both a closed and an open ball or can not be a ball.
  \item [(5)] If $(A,B)$ has a type $(\infty,1)$, then the corresponding ball $B_c$ of some realization $(X,d)$ can be both closed and open or only open.
  \item [(6)] If $(A,B)$ has a type $(1,\infty)$, then the corresponding ball $B_c$ of some realization $(X,d)$ can be both closed and open or only closed.
  \item [(7)] If there exists a point $c \in X$ such that the set $\Sp{c}$ has no last element, then the set $X$ in some realization can be a ball, both open and closed, or can not be a ball.
\end{itemize}

In this connection an open problem arises:
\emph{Let $(X,\delta)$ be an ordinal space. For every ball $B_c\in \mathbf B_X$ belonging to the one of the cases \emph{(4)--(7)} assign an admissible type of ball described in these cases. Does there exist a realization $(X,d)$ of $(X,\delta)$ such that the assigned types of balls coincide with types of balls (open, closed, not a ball) in $(X,d)$? }

\begin{exa}\label{ex4.1}
Let $(x_n)_{n=1}^{\infty}$ and $(y_n)_{n=1}^{\infty}$ be
 strictly decreasing sequences of positive real numbers such that
$\lim_{n\to\infty}x_n=0$ and
$\lim_{n\to\infty}y_n=p>0$. Let $X=  \{x_0,x_1,\ldots ,x_n,\ldots \}$ and
$Y=  \{y_0, y_1,\ldots ,y_n,\ldots \}$, where $x_0=0$ and $y_0=p$. Define metrics $d_X$ and $d_Y$ by the rules
\begin{equation}\label{eq4.1}
\begin{split}
d_X(x_i,x_j)&=
\begin{cases}
\max\{x_i, x_j\},  &x_i\neq x_j;\\
0, &x_i= x_j;
\end{cases}\\
d_Y(y_i,y_j)&=
\begin{cases}
\max\{y_i, y_j\},  &y_i\neq y_j;\\
0, &y_i= y_j.
\end{cases}
\end{split}
\end{equation}
It can be proved directly that $(X,d_X)$, $(Y,d_Y)$ are ultrametric spaces.
The functions $\Phi$ and $f$ defined by
\begin{equation}\label{eq4.2}
f(0)=  0,\quad f(x_i)=  y_i, \quad \Phi(x_0)=  y_0,\quad
\Phi(x_i)=  y_i,\quad \, \ i=1,2,\ldots
\end{equation}
are bijective and, moreover, $f$ is increasing. It follows from~(\ref{eq4.1}) and~(\ref{eq4.2}) that equation~(\ref{e2.1}) holds for all $x_i, x_j \in X$. Consequently, we have that $(X,d_X)$ and $(Y,d_Y)$ are weakly similar with the realization $(f,\Phi)$.
The spaces $X$ and $Y$ are not homeomorphic,
because $X$ has the limit point $x_0$ but $Y$ is discrete. Nevertheless, according to Proposition~\ref{p23} their ordinal types $(X,\delta)$ and $(Y,\delta)$ are isomorphic.
One can see that
$$
\widetilde{\mathbf B}_{X}=\bar{\mathbf B}_{X}=\{\{x_i\}\, | \, i\in \mathbb N^+\}\cup\{\{x_0\}\cup\{x_k, x_{k+1},\ldots \} \, | \,  k\in \mathbb N^+\}
$$
and
$$
\widetilde{\mathbf B}_{Y}=\bar{\mathbf B}_{Y}=\{\{y_i\} \, |\, i\in \mathbb N^+\}\cup \{\{y_0\}\cup\{y_k, y_{k+1},\ldots \} \, | \,  k\in \mathbb N^+\}\cup\{\{y_0\}\}.
$$
Here $\NN^+=\{1,2,3,\ldots\}$.
In other words, there exists an ordinal space $(X,\delta)$ with nonhomeomorphic realizations $(X,d_X)$ and $(Y,d_Y)$ and with different sets of balls $\widetilde{\mathbf B}_{X}=\bar{\mathbf B}_{X}$ and $\widetilde{\mathbf B}_{Y}=\bar{\mathbf B}_{Y}$ (speaking about realization $(Y,d_Y)$ we identify $y_i$ with $x_i$).
\end{exa}

\section{Topological properties}\label{topology}

\subsection{Bases of topology}
In this subsection we give a condition under which an appropriate set of balls in an ordinal space forms a base of some topology, which is natural in view of the following basic topological property of metric spaces:
\begin{prop}\label{p7}
The set of open balls of a metric space $X$ is the base of a topology on $X$ which is called the metric topology.
 \end{prop}

Note that a standard proof of Proposition~\ref{p7} is based on the following well-known proposition which we need below.

\begin{prop}\label{p7.1}
Let $\mathcal B$ be a family of subsets of a set $X$ such that the following conditions hold:
\begin{itemize}
  \item [$(a)$] every point $x\in X$ belongs to some $B\in \mathcal B$,
  \item [$(b)$] if $x\in B_1\cap B_2$ and $B_1, B_2\in \mathcal B$, then there exists a $B_3 \in \mathcal B$ such that $x\in B_3 \subseteq B_1\cap B_2$.
\end{itemize}
Then $\mathcal B$ is a base of some (uniquely defined) topology on the set $X$.
\end{prop}

For the set of all accumulation points of a semimetric space $X$ we use the designation $\ac(X)$. The following theorem is a generalization of Proposition~\ref{p7} to semimetric spaces.

\begin{thm}\label{t7.1}
Let $(X,d)$ be a semimetric space. If for every $z\in \ac(X)$ there exists an $\ve>0$ such that the inequality
\begin{equation}\label{e7.3}
      d(x,y)\leqslant d(x,z)+d(z,y)
\end{equation}
holds  for all $x\in X$ and all $y\in B_{\ve}(z)$, then $\widetilde{\mathbf B}_X$ is a base of some topology on $X$.
\end{thm}
\begin{proof}
To prove this theorem it suffices to establish conditions $(a)$ and $(b)$ of Proposition~\ref{p7.1} for $\mathcal B = \widetilde{\mathbf B}_X$. Condition $(a)$ evidently holds. Let $z\in B_1\cap B_2$, $B_1, B_2 \in \widetilde{\mathbf B}_X$. If $z$ is an isolated point, then $B_3$ can be chosen as $B_3=\{z\}$. Otherwise, suppose $z\in \ac(X)$ and $B_1=B_{r_1}(b_1)$ and $B_2=B_{r_2}(b_2)$ for some $b_1,b_2\in X$ and $r_1,r_2\in {\mathbb R}$. Let $\ve_0=\min\{r_1-d(b_1,z), r_2-d(b_2,z)\}$ and let  $\ve \leqslant \ve_0$ be a positive real such that~(\ref{e7.3}) holds for all $x\in X$ and all $y$ with $d(z,y)<\ve$. Then
$$
d(b_i,y)\leqslant d(b_i,z)+d(z,y)<d(b_i,z)+\ve \leqslant d(b_i,z)+r_i-d(b_i,z)=r_i.
$$
Consequently, $y\in B_i$, $i=1,2$, hence $y\in B_{\ve}(z)\subseteq B_1\cap B_2$.
\end{proof}

\begin{prop}\label{p7.3}
Let $(X,\delta)$ be an ordinal space and let $\mathcal B$ be some set of balls in $(X,\delta)$.
If there exists a semimetric $d$ on $X$, with the property that for every $z\in \ac(X,d)$ there exists an $\ve>0$ such that inequality~\emph{(\ref{e7.3})} holds  for all $x\in X$ and all $y\in B_{\ve}(z)$ and $\mathcal B = \widetilde{\mathbf B}_{(X,d)}$, then $\mathcal B$ is a base of some topology on $(X,\delta)$.
\end{prop}
\begin{proof}
The proof follows directly from Proposition~\ref{p7.1} and Theorem~\ref{t7.1}.
\end{proof}

\subsection{Connections with uniform spaces}

In this subsection it is shown that given an ordinal space $(X,\delta)$, the set $X$ can be equipped with
a topology in a natural way and, under an additional assumption, with an uniform structure. For an introduction to uniform spaces see, for example,~\cite{Is64} or~\cite[pp.~169--218]{B}.

Let $(X,L,\delta)$ be an ordinal space such that there is no least element in $L\setminus \{0\}$. Define a system $\mfu$ of subsets of $X\times X$ by $U\in \mfu$ iff there exists a cut $(A,B)$ of the set $L$ with $A\setminus \{0\}\neq \varnothing$, such that the implication
$$
(\delta(u,v)\in A)\Rightarrow ((u,v) \in U)
$$
holds for all $u,v \in X$.

It is clear that the following is valid:
\begin{align}
\mbox{$\Delta = \{(x,x)\, | \, x \in X\}\subseteq U$ for all $U\in \mfu$,}\tag{U1}\label{u1}\\
\mbox{If $U\in \mfu$ and $U\subset V \subseteq X\times X$, then $V\in \mfu$,}\tag{U2}\label{u2}\\
\mbox{If $U,V \in \mfu$, then $U\cap V \in \mfu$,}\tag{U3}\label{u3}\\
\mbox{If $U\in \mfu$, then $\{(y,x) \, | \, (x,y)\in U\}\in \mfu$.\tag{U4}\label{u4}}
\end{align}

Recall that a \emph{uniformity} on a set $X$ is a structure given by a non-empty set $\mfu$ of subsets of $X\times X$ which satisfies axioms  \eqref{u1}--\eqref{u4}, and axiom \eqref{b4} from Definition~\ref{d6.3} below with
$\mfu=\mfb$, see~\cite[p.~169]{B}. Also note that axioms \eqref{u1}--\eqref{u3} mean that $\mfu$ is a filter if $X\neq\varnothing$.

For $U\in \mfu$ and $x\in X$, let $U[x]=\{y\in X \, | \, (x,y) \in U\}$. Define the set $\mathfrak O$
of subsets $O$ of $X$ as follows: $O\in \mathfrak O$ iff for each $x\in O$ there exists some $U\in \mfu$ with $U[x]\subseteq O$.
\begin{prop}~\label{p314}
  $\mathfrak O$ is a topology on $X$.
\end{prop}
\begin{proof}
It is clear that $\varnothing, X\in \mathfrak O$ and any union of elements of $\mathfrak O$ is an element of $\mathfrak O$. Suppose $O_1, O_2 \in \mfo$. Let us show that $O_1\cap O_2 \in \mfo$ and let $x\in O_1\cap O_2$. By definition there exist $U_1, U_2\in \mfu$ such that $U_1[x]\subset O_1, U_2[x]\subset O_2$. Let $(A_1,B_1)$, $(A_2, B_2)$ be the corresponding cuts for $U_1$ and $U_2$. Without loss of generality consider that $A_1\subset A_2$. Let
$$
\widetilde{U}_1=\{(u,v)\, | \, \delta(u,v)\in A_1\},\quad \widetilde{U}_2=\{(u,v)\, | \, \delta(u,v)\in A_2\}.
$$
It is clear that $\widetilde{U}_1, \widetilde{U}_2 \in \mfu$, $\widetilde{U}_1\subset U_1$, $\widetilde{U}_1\subset \widetilde{U}_2$, $\widetilde{U}_2\subset U_2$, i.e., $\widetilde{U}_1\subset U_1, U_2$. Hence $\widetilde{U}_1[x]\subset U_1[x], U_2[x]$. Since $\widetilde{U}_1[x]\subset U_1[x]\cap U_2[x]\subset O_1\cap O_2$,
we have $O_1\cap O_2 \in \mfo$.
\end{proof}

If $(X,L,\delta)$ is an ordinal space such that there exists a least element in $L\setminus \{0\}$, then it is natural
to consider the discrete topology on $X$.

\begin{defn}[{\protect\cite[p.~170]{B}}]\label{d6.3}
 A non-empty set $\mfb$ of subsets of $X\times X$ is a fundamental system of entourages of a uniformity on $X$ if and only if $\mfb$  satisfies the following axioms:
\begin{align}
\mbox{The intersection of two sets of $\mfb$ contains a set of $\mfb$,}\tag{B1}\label{b1}\\
\mbox{Every set of $\mfb$ contains the diagonal $\Delta$,}\tag{B2}\label{b2}\\
\mbox{For each $V\in \mfb$ there exists $V'\in \mfb$ such that $V'\subset V^{-1}$,}\tag{B3}\label{b3}\\
\mbox{For each $V\in \mfb$ there exists $W\in \mfb$ such that $W^2\subset V$.}\tag{B4}\label{b4}
\end{align}
The uniformity defined by $\mfb$ consists of all subsets of $X$ containing some $B\in\mfb$. (One easily sees that indeed a uniformity is defined.)
\end{defn}

We recall that if $W$ is a subset of $X\times X$, then the set of pairs $(x,y) \in X\times X$ such that $(x,z), (z,y)\in W$ for some $z\in X$ is denoted by $W^2$, and that the set of pairs $(x,y) \in X\times X$ such that $(y,x)\in V$ is denoted by $V^{-1}$.

For a metric space $(X,d)$ and, more generally, for a pseudometric space a fundamental system of entourages, hence a uniformity, is provided by the sets \begin{equation}\label{s61}
U_a=\{(x,y) \in X\times X\colon d(x,y)\leqslant a\}
\end{equation}
with $a>0$. One of the approaches for defining uniform structures is based on using systems of pseudometrics,
where a uniformity is given by the least upper bound of the uniform structures defined by the single pseudometrics.

In the next proposition we show that uniform structures can be defined by a wider class of spaces than pseudometric ones, namely by spaces for which the triangle inequality need not hold for all triples of points.

We shall say that a pair $(X,d)$ is a \emph{pseudosemimetric space} if $X$ is a set and $d\colon X\times X\to\mathbb R^+$ is a mapping satisfying the properties $d(x,y)=d(y,x)$ and $d(x,x)=0$ for all $x,y \in X$.
\begin{thm}\label{p34}
Let $(X,d)$ be a pseudosemimetric space and let there exist an $\ve>0$ such that for every $z\in X$ the triangle inequality
\begin{equation}\label{ti6}
d(x,y)\leqslant d(x,z)+d(z,y)
\end{equation}
holds for every $x,y \in B_{\ve}(z)$. Then the sets~\emph{(\ref{s61})} form a fundamental system of entourages of some uniformity on $X$.
\end{thm}
\begin{proof}
For $a\geqslant b>0$ it holds $U_a\cap U_b =U_b$. Hence axiom \eqref{b1} is satisfied. Axiom \eqref{b2} is valid because the equality $d(x,x)=0$ holds for every $x\in X$. Since $d$ is symmetric, we have $U_a^{-1}=U_a$ and therefore \eqref{b3} is satisfied. Let $U_a$ be an entourage. Define $b$ as follows:
$$
b=
\begin{cases}
  \frac{\ve}{2}, &a>\ve;\\
  \frac{a}{2}, &a\leqslant \ve.
\end{cases}
$$
Observe that for every $U_a$ with $a\leqslant \ve$ by inequality~(\ref{ti6}) we have the inclusion $U_a^2\subseteq U_{2a}$. Hence for the case $a>\ve$ we have the relations
$$
U_b^2=U_{\frac{\ve}{2}}^2\subset U_{\ve}\subset U_a,
$$
and for the case $a\leqslant \ve$ we have
$$
U_b^2=U_{\frac{a}{2}}^2\subset U_a,
$$
which establish axiom \eqref{b4}.
\end{proof}

\section{Isomorphisms of Hasse diagrams}\label{Hasse}

\noindent Recall some definitions. Let  $(P,\leqslant)$ be a finite partially ordered set. The \emph{Hasse diagram} $\mathcal{H}(P)$ of the poset $(P,\leqslant)$ is a directed graph with the set of vertices $P$ and the set of arcs (directed edges) $A_P\subseteq P\times P$ such that the pair $\langle u,v \rangle$ belongs to $A_P$ if and only if $u\leqslant v$, $u\neq v$ and the implication
$$
(u\leqslant w\leqslant v)\Rightarrow (u=w \vee v=w)
$$
holds for every $w\in P$.

Two directed graphs $(X,A_X)$ and $(Y,A_Y)$ are isomorphic, if there exists a bijection $F\colon X\to Y$ such that
$$
  (\langle x,y\rangle\in A_X)\Leftrightarrow(\langle F(x),F(y)\rangle \in A_Y).
$$
In this case $F$ is an isomorphism of the directed graphs $(X,A_X)$ and $(Y,A_Y)$.

It is evident that for every ordinal space $X$ the set $\mathbf B_X$ can be considered as a poset $(\mathbf B_X,\subseteq)$ with the partial order defined by the relation of inclusion $\subseteq$.

Hasse diagrams are a very convenient tool for describing ball-structure of ordinal spaces. If node $C$ is a successor of $A$ in $\mathcal H(\mathbf B_X)$, then the ball $A\in \mathbf B_X$ is contained in $C\in \mathbf B_X$ and $A\neq C$. If additionally $C$ is a direct successor of $A$, then there is no ball between $A$ and $C$, i.e., the relation $A\subseteq B \subseteq C$ is impossible for any $B\in \mathbf B_X$ with $A\neq B \neq C$. Otherwise the balls $A$ and $C$ are not comparable.

Let  $(X,\delta_X)$  and $(Y,\delta_Y)$ be ordinal spaces. A mapping $F\colon X\to Y$ is called \emph{ball-preserving}, if for every $Z\in \mathbf B_X$ and $W \in \mathbf B_Y$ it holds
$$
F(Z) \in \mathbf B_Y \text{ and } F^{-1}(W) \in \mathbf B_X.
$$

The following theorem is a generalization of a series of results. In~\cite{P(TIAMM)} it was shown that representing trees of finite ultrametric spaces are isomorphic if and only if there exists a ball-preserving bijection between these spaces. In fact the representing tree of a finite ultrametric space is a Hasse diagram of the ballean, i.e., the set of balls of this space. Later it became clear that this assertion also holds  for metric spaces and moreover for semimetric ones~\cite[Theorem 3.3]{P18}. The proof of Theorem~\ref{t25} is analogous to the proof of Theorem 3.3 from \cite{P18} but  we reproduce it here for convenience.

\begin{thm}\label{t25}
Let $X$ and $Y$ be finite ordinal spaces. Then the Hasse diagrams $\mathcal H(\mathbf B_X)$ and $\mathcal H(\mathbf B_Y)$ are isomorphic as directed graphs if and only if there exists a bijective ball-preserving mapping $f\colon X\to Y$.
\end{thm}

\begin{proof}
Suppose first that there exists a map $f\colon X\to Y$ which is bijective and ball-preserving. Let $\Phi\colon \mathbf B_X \to \mathbf B_Y$ be the mapping defined in the following way:
\begin{equation*}
\mathbf{B}_X \ni B \stackrel{\rm \Phi}{\longmapsto} f(B) \in \mathbf{B}_Y.
\end{equation*}
It is easy to see that $\Phi$ is a bijection with $\Phi$ and $\Phi^{-1}$ order-preserving, i.e., that $\Phi$ is an order isomorphism. To prove that $\mathcal H(\mathbf B_X)$ and $\mathcal H(\mathbf B_Y)$ are isomorphic with isomorphism $\Phi$ we have to establish the following equivalence:
\begin{equation}\label{equiv}
\langle B_1,B_2 \rangle \in A_{\mathcal H(\mathbf B_X)} \Leftrightarrow \langle \Phi(B_1),\Phi(B_2) \rangle \in A_{\mathcal H(\mathbf B_Y)}
\end{equation}
for all $B_1,B_2\in \mathbf B_X$.

According to the definition of a Hasse diagram the left part of equivalence~(\ref{equiv}) is equivalent to the following two conditions:
\begin{itemize}
  \item [A)] $B_1\subset B_2$.
  \item [B)] ($B_1\subseteq B \subseteq B_2) \Rightarrow (B_1= B \vee B_2=B)$ for every $B\in \mathbf B_X$.
\end{itemize}
Analogously, the right part of equivalence~(\ref{equiv}) is equivalent to
\begin{itemize}
  \item [C)] $\Phi(B_1)\subset \Phi(B_2)$.
  \item [D)] ($\Phi(B_1)\subseteq \tilde{B} \subseteq \Phi(B_2)) \Rightarrow (\Phi(B_1)= \tilde{B} \vee \Phi(B_2)=\tilde{B})$ for every $\widetilde{B}\in \mathbf B_Y$.
\end{itemize}
The equivalence of conditions A) and C) follows directly from the fact that $\Phi$ is an order isomorphism between $\mathbf B_X$ and $\mathbf B_Y$. Suppose that the implication B) $\Rightarrow$ D) does not hold. Consequently, there exists a $\tilde{B}\in \mathbf B_Y$ such that $\Phi(B_1)\subseteq \tilde{B} \subseteq \Phi(B_2)$ and $\Phi(B_1)\neq \tilde{B} \neq \Phi(B_2)$. Since $\Phi$ is an order isomorphism, we have $B_1\subseteq \Phi^{-1}(\tilde{B})\subseteq B_2$ and   $B_1 \neq \Phi^{-1}(\tilde{B}) \neq B_2$ which contradicts condition B). The implication D) $\Rightarrow$ B) can be proved analogously.

Conversely, suppose that $\mathcal H(\mathbf B_X)\simeq \mathcal H(\mathbf B_Y)$, with the isomorphism $F\colon \mathbf B_X\to\mathbf B_Y$. Define by $B_X^1$ and $B_Y^1$ the sets of one-point balls of the spaces $X$ and $Y$. It is clear that the vertices of $\mathcal H(\mathbf B_X)$ and $\mathcal H(\mathbf B_Y)$ having no incoming arc
form the sets $B_X^1$ and $B_Y^1$, respectively. Since $F$ is an isomorphism, the mapping
\begin{equation*}
f^*=F|_{B_X^1}\colon B_X^1\to B_Y^1
\end{equation*}
is a bijection. Define a mapping $f\colon X\to Y$ by
$$
(f(x)=y) \ \, \Leftrightarrow \ \, (f^*(\{x\})=\{y\}),
$$
which is evidently a bijection. We claim that $f$ is ball-preserving.

It is clear that for $B\in \mathbf B_X$ the image $f(B)$ is a ball in $Y$ if and only if the equality
\begin{equation}\label{eq11}
f^*\left(\bigcup\limits_{x\in B}\{\{x\}\}\right) = \bigcup\limits_{y \in B'} \{\{y\}\}
\end{equation}
holds  for some $B'\in \mathbf B_Y$.

Let $B$ be a ball in $X$. Define by $G_B$ the directed subgraph of $\mathcal H(\mathbf B_X)$ induced by the set of all predecessors of $B\in V(\mathcal H(\mathbf B_X))$ and by $G_{F(B)}$ the directed subgraph of $\mathcal H(\mathbf B_Y)$ induced by the set of all predecessors of $F(B) \in V(\mathcal H(\mathbf B_Y))$. Since $\mathcal H(\mathbf B_X) \simeq \mathcal H(\mathbf B_Y)$ and $F$ is an isomorphism, we have
\begin{equation}\label{eq12}
G_B\simeq G_{F(B)}.
\end{equation}

For all  $Z\in  V(\mathcal H(\mathbf B_X))$ ($W\in  V(\mathcal H(\mathbf B_Y))$)  denote by $\Gamma_X(Z)$ ($\Gamma_Y(W)$) the set of all predecessors of $Z$  ($W$) with no incoming arc.
By~(\ref{eq12}) we have $f^* (\Gamma_X(B)) = \Gamma_Y (F(B))$. To establish~(\ref{eq11}) it suffices to note that
$$
\Gamma_X(B)=\bigcup\limits_{x\in B}\{\{x\}\}  \ \text{ and  } \ \Gamma_Y(F(B))=\bigcup\limits_{y\in F(B)}\{\{y\}\}.
$$
These two equalities follow directly from the constructions of $\mathcal H(\mathbf B_X)$ and $\mathcal H(\mathbf B_Y)$, respectively. The arguing for $f^{-1}$ is analogous.
\end{proof}

The following statement is obvious:

\begin{prop}\label{p26}
For finite isomorphic ordinal spaces $(X,\delta_X)$  and $(Y,\delta_Y)$ the Hasse diagrams $\mathcal{H}(\mathbf B_X)$ and $\mathcal{H}(\mathbf B_Y)$ are isomorphic as directed graphs.
\end{prop}

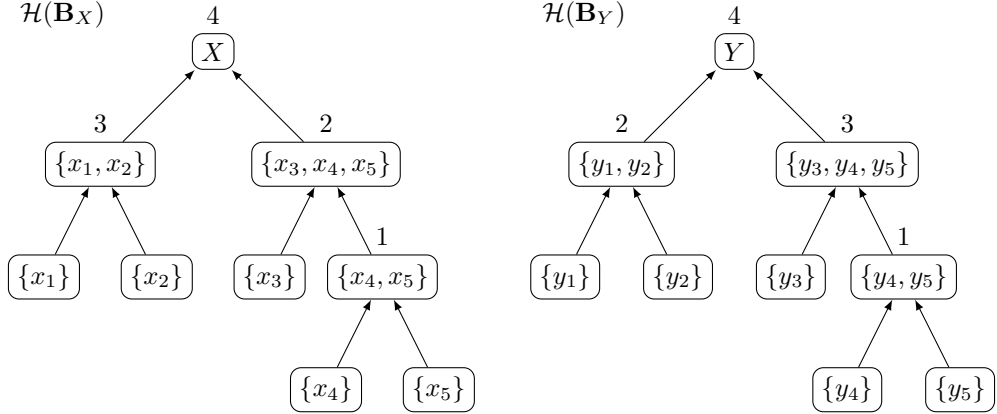
\begin{figure}[ht]
\begin{minipage}[h]{0.49\linewidth}
\begin{center}
\begin{tikzpicture}[sibling distance=5em,
  every node/.style = {shape=rectangle, rounded corners,
    draw, align=center}]
\tikzstyle{level 1}=[level distance=15mm,sibling distance=30mm]
\tikzstyle{level 2}=[level distance=15mm,sibling distance=15mm]

\draw (-2,0.5) node[draw=none] {$\mathcal H(\mathbf B_X)$};

\node[label=above:{$4$}]{$X$}
    child [>=latex, <-]{
        node[label=above:{$3$}]{$\{x_1,x_2\}$}
            child [>=latex, <-] { node{$\{x_1\}$} }
            child [>=latex, <-] { node {$\{x_2\}$} }
    }
    child [>=latex, <-] {
        node[label=above:{$2$}] {$\{x_3, x_4, x_5\}$}
            child[>=latex, <-] { node {$\{x_3\}$} }
            child [>=latex, <-]{
                node[label=above:{$1$}] {$\{x_4, x_5\}$}
                    child[>=latex, <-] { node {$\{x_4\}$} }
                    child [>=latex, <-]{ node {$\{x_5\}$} }
        }
    };
\end{tikzpicture}
\end{center}
\end{minipage}
\hfill
\begin{minipage}[h]{0.49\linewidth}
\begin{center}
\begin{tikzpicture}[sibling distance=5em,
  every node/.style = {shape=rectangle, rounded corners,
    draw, align=center}]
\tikzstyle{level 1}=[level distance=15mm,sibling distance=30mm]
\tikzstyle{level 2}=[level distance=15mm,sibling distance=15mm]

\draw (-2,0.5)node[draw=none] {$\mathcal H(\mathbf B_Y)$};

   \node[label=above:{$4$}]{$Y$}
    child [>=latex, <-] {
        node[label=above:{$2$}]{$\{y_1,y_2\}$}
            child[>=latex, <-] { node{$\{y_1\}$} }
            child[>=latex, <-] { node {$\{y_2\}$} }
    }
    child[>=latex, <-] {
        node[label=above:{$3$}] {$\{y_3, y_4, y_5\}$}
            child[>=latex, <-] { node {$\{y_3\}$} }
            child[>=latex, <-] {
                node[label=above:{$1$}] {$\{y_4, y_5\}$}
                    child[>=latex, <-] { node {$\{y_4\}$} }
                    child[>=latex, <-] { node {$\{y_5\}$} }
        }
    };
\end{tikzpicture}
\end{center}
\end{minipage}
\caption{The Hasse diagrams $\mathcal H(\mathbf B_X)$ and $\mathcal H(\mathbf B_Y)$ for the ordinal spaces  $(X,\delta_X)$, $(Y,\delta_Y)$. }
\label{f1}
\end{figure}

In what follows by $l(v)$ we denote a label of a vertex $v$ belonging to some graph and by $\operatorname{LCS}(v,w)$ the least common successor of the vertices $v$ and $w$ of some directed graph.
\begin{rem}
The assertion converse to Proposition~\ref{p26} is not true. Let $X=\{x_1,x_2, x_3, x_4, x_5\}$ and  $Y=\{y_1,y_2, y_3, y_4, y_5\}$ be the sets of leaves of the directed trees depicted at Figure~\ref{f1}. Define metric spaces $(X,d_X)$, $(Y,d_Y)$ by the following rule:
$$
d_X(x_i, x_j)=
\begin{cases}
l(\operatorname{LCS}(\{x_i\},\{x_j\})), &x_i\neq x_j;\\
0, &x_i=x_j,
\end{cases}
$$
$$
d_Y(y_i, y_j)=
\begin{cases}
l(\operatorname{LCS}(\{y_i\},\{y_j\})), &y_i\neq y_j;\\
0, &y_i=y_j,
\end{cases}
$$
and let $(X,\delta_X)$, $(Y,\delta_Y)$ be their ordinal types, respectively. It is easy to show that $(X,d_X)$ and $(Y,d_Y)$ are ultrametric spaces which are not weakly similar. By Proposition~\ref{p23} the ordinal spaces $(X,\delta_X)$, $(Y,\delta_Y)$ are not isomorphic, though the Hasse diagrams $\mathcal H(\mathbf B_X)$, $\mathcal H(\mathbf B_Y)$ of ordinal spaces  $(X,\delta_X)$, $(Y,\delta_Y)$ are isomorphic as directed graphs and coincide with the directed trees at Figure~\ref{f1}.

Note that it was proved in~\cite[Theorem 2.8]{DPT} that for a finite metric space $X$ the Hasse diagram $\mathcal H(\mathbf B_X)$  is a tree if and only if $X$ is an ultrametric space.
\end{rem}

With $\mathbf B_X=\bar{\mathbf B}_X=\widetilde{\mathbf B}_X$, justified by coincidence of open and closed balls in finite semimetric spaces, one easily sees the following:

\begin{prop}\label{c4.8}
Let $(X,d)$ be a realization of the finite ordinal space $(Y,\delta_Y)$. Then the Hasse diagrams $\mathcal{H}(\mathbf B_X)$ and $\mathcal{H}(\mathbf B_Y)$ are isomorphic as directed graphs.
\end{prop}

\section{The number of balls in finite ordinal spaces}\label{ballnum}
\noindent In this section we formulate two conjectures about the maximal and the minimal number of balls in finite ordinal spaces $(X,L,\delta)$. In the latter case we use an additional condition that $\delta(x,y)\neq \delta(z,w)$ for every different pairs of points $\{x,y\}\neq\{z,w\}$ with $x\neq y$, $z\neq w$. Without this additional condition we trivially obtain that the minimal number of balls is equal to $|X|+1$ if $|X|>1$.

\begin{con}\label{cofin1}
Let $X$ be an ordinal space with $|X|=n\geqslant 1$. Then the maximal number of balls in $X$ is equal to $b_n$, where
$$
b_1=1, \  b_2=3, \  b_3=6, \  b_4=12, \  b_5=19, \  b_6=29, \  b_7=40, \ldots
$$
is the sequence A263511 from~\textup{\cite{oeis}}  (Total number of ON (black) cells after n iterations of the ``Rule 155'' elementary cellular automaton starting with a single ON (black) cell.).
\end{con}

Conjecture \ref{cofin1} and Conjecture \ref{cofin2} below are based on direct constructions of ordinal spaces having a maximal number of balls and a minimal number of balls with the above mentioned restrictions, respectively. In the next example we  construct an ordinal space $X$,  $|X|=6$, with the greatest possible number of balls. Note also that such a construction is not a proof of this maximality.

\begin{exa} Let $X=\{a,b,c,d,e,f\}$.  We describe the relations on $X\times X$ by labeling all the unordered pairs of points from $X$ by natural numbers from the set $\{1,2,3,\ldots ,14,15\}$ such that $\delta(x,y)<\delta(z,w)$ if and only if the label of the pair $\{x,y\}$ is strictly less then the label of $\{z,w\}$. See the symmetric table~(\ref{em}) with the labels given in dependence of points corresponding to a row and column, respectively.

\begin{equation}\label{em}
    \begin{array}{c|cccccc}
        & a & b & c & d & e & f \\ \hline
      a & \cdot & 5 & 12 & 15 & 8 & 6 \\
      b & 5 & \cdot & 4 & 11 & 14 & 7 \\
      c & 12 & 4 & \cdot & 3 & 10 & 13 \\
      d & 15 & 11 & 3 & \cdot & 2 & 9 \\
      e & 8 & 14 & 10 & 2 & \cdot & 1 \\
      f & 6 & 7 & 13 & 9 & 1 & \cdot \\
    \end{array}
\end{equation}
Every of the following rows contains a list of all balls with the centers $a,b,c,\ldots $, respectively:
\begin{equation*}
    \begin{array}{ccccccc}
\{a\}, & \{a,b\}, & \{a,b,f\}, & \{a,b,e,f\}, & \{a,b,c,e,f\}, & X; \\

\{b\}, & \{b,c\}, & \{a,b,c\}, & \{a,b,c,f\}, & \{a,b,c,d,f\}, & X; \\

\{c\}, & \{c,d\}, & \{b,c,d\}, & \{b,c,d,e\}, & \{a,b,c,d,e\}, & X; \\

\{d\}, & \{d,e\}, & \{c,d,e\}, & \{c,d,e,f\}, & \{b,c,d,e,f\}, & X; \\

\{e\}, & \{e,f\}, & \{d,e,f\}, & \{a,d,e,f\}, & \{a,c,d,e,f\}, & X; \\

\{f\}, & \{e,f\}, & \{a,e,f\}, & \{a,b,e,f\}, & \{a,b,d,e,f\}, & X.
    \end{array}
\end{equation*}

The total number of different balls is $b_6=29$.
\end{exa}

\begin{con}\label{cofin2}
Let $(X,\delta)$ be an ordinal space with $|X|=n$ and let $\delta(x,y)\neq \delta(z,w)$ for every different pairs of points $\{x,y\}\neq\{z,w\}$ with $x\neq y$, $z\neq w$. Then the following inequality holds:
\begin{equation}\label{e5.1}
\frac{n(n+1)}{2} \leqslant |\mathbf B_X|.
\end{equation}
Equality in~\textup{(\ref{e5.1})} holds if and only if the Hasse diagram $\mathcal H(X)$ has a structure as depicted in Figure~\textup{\ref{f2}}.
\end{con}

\begin{rem}
Note that a number of the form $n(n+1)/2$ for some $n\geqslant 1$ is called a triangular number since it counts the number of points which can be arranged to an ``equilateral triangle'' of side length $n$.
\end{rem}
\begin{figure}[ht]
\begin{minipage}[h]{0.39\linewidth}
\begin{center}
\begin{tikzpicture}[sibling distance=5em,
  every node/.style = {shape=rectangle, rounded corners,
    draw, align=center}]
\tikzstyle{level 1}=[level distance=15mm,sibling distance=15mm]
\tikzstyle{level 2}=[level distance=15mm,sibling distance=15mm]

\draw (-2,0.5)node[draw=none] {$\mathcal H(\mathbf B_X)$};

   \node {$X$}
    child [>=latex, <-] {
        node{$\{x_1,x_2\}$}
            child[>=latex, <-] { node{$\{x_1\}$} }
            child[>=latex, <-] { node{$\{x_2\}$} }
    }
    child[>=latex, <-] {
        node {$\{x_2, x_3\}$}
            child[>=latex, <-] { node {\phantom{$\{x_2\}$}} }
            child[>=latex, <-] { node {$\{x_3\}$} }
    };
\end{tikzpicture}
\end{center}
\end{minipage}
\hfill
\begin{minipage}[h]{0.59\linewidth}
\begin{center}
\begin{tikzpicture}[sibling distance=5em,
  every node/.style = {shape=rectangle, rounded corners,
    draw, align=center}]
\tikzstyle{level 1}=[level distance=15mm,sibling distance=20mm]
\tikzstyle{level 2}=[level distance=15mm,sibling distance=20mm]

\draw (-2,0.5)node[draw=none] {$\mathcal H(\mathbf B_Y)$};

   \node {$Y$}
    child [>=latex, <-] {
        node{$\{y_1,y_2,y_3\}$}
            child[>=latex, <-] { node{$\{y_1, y_2\}$}
                   child[>=latex, <-] { node {$\{y_1\}$} }
                    child [>=latex, <-]{ node(2) {$\{y_2\}$} }
            }
            child[>=latex, <-] { node{$\{y_2, y_3\}$}}
    }
    child[>=latex, <-] {
        node {$\{y_2, y_3,y_4\}$}
            child[>=latex, <-] { node(1) {\phantom{$\{y_2, y_3\}$}} }
            child[>=latex, <-] { node {$\{y_3, y_4\}$}
                    child[>=latex, <-] { node(3) {$\{y_3\}$} }
                    child [>=latex, <-]{ node {$\{y_4\}$} }
            }
    };

          \draw[>=latex,->] (2)--(1);
        \draw[>=latex,->] (3)--(1);
\end{tikzpicture}
\end{center}
\end{minipage}
\caption{The Hasse diagrams $\mathcal H(\mathbf B_X)$, $\mathcal H(\mathbf B_Y)$ of the ordinal spaces  $(X,\delta_X)$, $(Y,\delta_Y)$ having minimal numbers of balls with $|X|=3$, $|Y|=4$.}
\label{f2}
\end{figure}
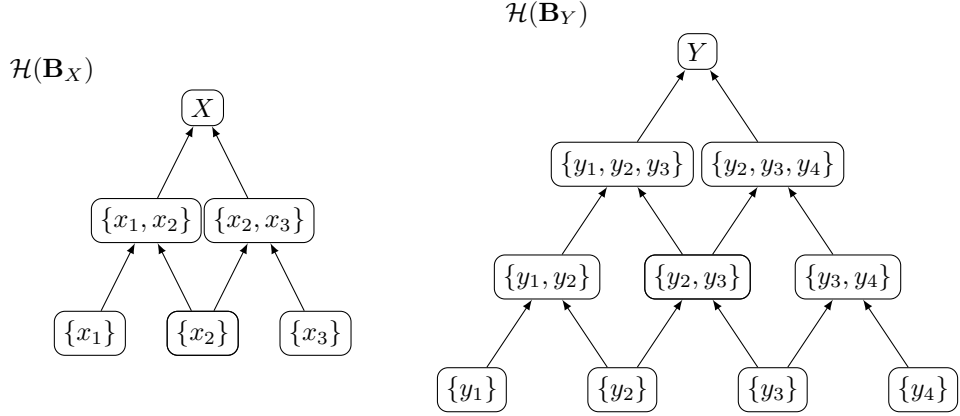

\begin{rem}
Of course Conjectures \ref{cofin1} and \ref{cofin2} could be formulated for metric spaces instead of ordinal ones.
Indeed, the considered problems do not depend on the triangle inequality. Moreover, we saw above that also the concept of distance is not important here.  From this point of view, ordinal spaces provide a “minimal structure” for the problems considered.
\end{rem}

\section{Embeddings of ordinal spaces in the real line}\label{rline}
\noindent We shall say that an ordinal space $(X, \delta)$ is \emph{embeddable} in $\mathbb R^n$ if there exists a mapping
$\Phi\colon X\to \mathbb R^n$ such that the following relations hold for all $x,y,z,w \in X$:
\begin{equation}\label{e221}
\delta(x,y)<\delta(z,w) \text{ iff }  d(\Phi(x),\Phi(y))<d(\Phi(z),\Phi(w)),
\end{equation}
\begin{equation}\label{e222}
\delta(x,y)=\delta(z,w) \text{ iff }  d(\Phi(x),\Phi(y))=d(\Phi(z),\Phi(w)),
\end{equation}
\begin{equation}\label{e223}
\delta(x,y)>\delta(z,w) \text{ iff }  d(\Phi(x),\Phi(y))>d(\Phi(z),\Phi(w)),
\end{equation}
where $d$ is the Euclidean metric in $\mathbb R^n$.
If $X$ is embeddable in $\mathbb R^n$, then we write $X\hookrightarrow \mathbb R^n$. Note that~(\ref{e221}) implies that $\Phi$ is injective (if $z\neq w$, choose $x=y$ to get $\Phi(z)\neq \Phi(w)$).

It is easy to see that ordinal spaces $X$ with $|X|=1,2$ are always embeddable in $\mathbb R^1$ and a space $X$ with $|X|=3$ is embeddable in $\mathbb R^1$ if and only if its three points form a triangle with unique maximal side.

In order to now study embeddings of finite ordinal spaces $(X,\delta)$ in ${\mathbb R}^1$, we need some special concepts.  First of all, for $n\in \mathbb N^+$ let
\begin{eqnarray*}
{\mathcal I}_n=\{(i_0,i_1,\ldots , i_k)\in \{1,2,\ldots ,n\}^{k+1}\ |\ k\in{\mathbb N^+}, \, i_0\leq i_1\leq\ldots\leq i_{k-1}\leq i_k\}.
\end{eqnarray*}

\begin{defn}
Let  $x_1, x_2,\ldots ,x_n$ be an enumeration of the points of an ordinal space $X$ with $|X|=n$ and let $(i_0,i_1,\ldots , i_k), (j_0,j_1,\ldots , j_k)\in {\mathcal I}_n$. Then we write
\begin{equation*}
(i_0,i_1,\ldots , i_k)\sim (j_0,j_1,\ldots , j_k)
\end{equation*}
and
\begin{equation*}
(i_0,i_1,\ldots , i_k)\prec (j_0,j_1,\ldots , j_k),
\end{equation*}
respectively, if there exists a permutation $\pi$ of $\{1,2,\ldots ,k\}$ with
\begin{eqnarray}\label{allequal}
\delta (x_{i_{l-1}},x_{i_l})=\delta (x_{j_{\pi (l)-1}},x_{j_{\pi (l)}})\mbox{ for all }l\in\{1,2,\ldots ,k\}
\end{eqnarray}
and
\begin{eqnarray*}
\delta (x_{i_{l-1}},x_{i_l})\leq\delta (x_{j_{\pi (l)-1}},x_{j_{\pi (l)}})\mbox{ for all }l\in\{1,2,\ldots ,k\}\mbox{ but not \eqref{allequal},}
\end{eqnarray*}
respectively. (It is suggestive to think that $\pi$ associates to the $l$-th ``interval'' with the endpoints $x_{i_{l-1}},x_{i_l}$ on the left side the $\pi(l)$-th ``interval'' with the endpoints $x_{j_{\pi (l)-1}},x_{j_{\pi (l)}}$ on the right side for $l=1,2,\ldots ,k$.)

We say that the enumeration has the {\em majorization property} if for all $(i_0,i_1,\ldots , i_k),\linebreak (j_0,j_1,\ldots , j_k)\in {\mathcal I}_n$ and $k\in {\mathbb N^+}$ with
$(i_0,i_1,\ldots , i_k)\prec (j_0,j_1,\ldots , j_k)$ or $(i_0,i_1,\ldots , i_k)\sim (j_0,j_1,\ldots , j_k)$, respectively,
it holds $\delta(x_{i_0},x_{i_k})<\delta(x_{j_0},x_{j_k})$ or $\delta(x_{i_0},x_{i_k})=\delta(x_{j_0},x_{j_k})$, respectively.
\end{defn}

For the subsequent considerations, call a pair $\{a,b\}$ \emph{diametrical} in $(X,\delta)$ if $\delta(a,b)\geqslant \delta(x,y)$ for all $x,y \in X$.

Let us list some consequences of the majorization property.

\begin{prop}\label{p102}
Let $(X,\delta)$ be an ordinal space with $|X|=n\geqslant 2$ and $x_1, x_2,\ldots ,x_n$ be an enumeration of the points of $X$ satisfying the majorization property.
Then the following is valid:
      \begin{enumerate}
        \item[(i)] The implication
            $$
            (i\leqslant k \leqslant l \leqslant j) \Rightarrow (\delta(x_k, x_l)< \delta(x_i, x_j))
            $$
            holds for every $i,k,l,j \in \{1,\ldots ,n\}$ with $i\neq k$ or $l\neq j$.
        \item[(ii)] The following inequalities hold:
            $$
            \delta (x_1,x_2)<\delta (x_1,x_3)<\ldots <\delta (x_1,x_n)>\ldots>\delta(x_{n-2},x_n)>\delta(x_{n-1},x_n).
            $$
        \item[(iii)]  $\{x_1,x_n\}$ is a single diametrical pair in $X$.
        \item[(iv)] The equivalences
            \begin{eqnarray*}
            (\delta(x_i, x_j)<\delta(x_k, x_l))&\Leftrightarrow &(\delta(x_i, x_k)<\delta(x_j, x_l)),\\
             (\delta(x_i, x_j)=\delta(x_k, x_l))&\Leftrightarrow &(\delta(x_i, x_k)=\delta(x_j, x_l)),\\
              (\delta(x_i, x_j)>\delta(x_k, x_l))&\Leftrightarrow &(\delta(x_i, x_k)>\delta(x_j, x_l))
            \end{eqnarray*}
            hold for every $i,k,j,l \in \{1,\ldots ,n\}$ with $i<k<j<l$.
       \end{enumerate}
\end{prop}
\begin{proof}

(i): If $i\leqslant k \leqslant l \leqslant j$ and $i\neq k$ or $l\neq j$, it holds $(k,k,l,l)\prec (i,k,l,j)$, hence $\delta(x_k, x_l)<\delta(x_i,x_j)$

Conditions (ii) and (iii) easily follow from (i).

(iv): Let $i<k<j<l$. If $\delta(x_i, x_k)<\delta(x_j, x_l)$, then $(i,k,j)\prec (k,j,l)$, hence $\delta(x_i, x_j)<\delta(x_k, x_l)$. The converse implication can be established by contradiction. The other two equivalences are analogous.
\end{proof}

\begin{prop}\label{p101}
Let $(X,\delta)$ be an ordinal space with $|X|=n\in {\mathbb N^+}$ and $X\hookrightarrow \mathbb R^1$. Then there exists an enumeration $x_1, x_2,\ldots ,x_n$ of the points of the space $X$ satisfying the majorization property.
\end{prop}
\begin{proof} Let $\Phi$ be an embedding mapping for $X\hookrightarrow \mathbb R^1$.
It is easy to see that the enumeration $x_1,x_2,\ldots ,x_n$, such that $x_i$, $i=1,..,k$, are the preimages of the corresponding consecutive points in $\mathbb R^1$ under the mapping $\Phi$, satisfies the majorization property.
\end{proof}

We are not able to say whether the existence of an enumeration with majorization property is also sufficient for embeddability in $\mathbb R^1$ if $n>5$, but for $n\leq 4$ we have sufficiency.
This is trivial for $n<4$. As already mentioned, for $n=3$ the existence of a unique pair of diametrical points is equivalent to embeddability in $\mathbb R^1$.

\begin{con}
Let $(X,\delta)$ be an ordinal space with $|X|=n\in {\mathbb N^+}$ and let
there exist an enumeration $x_1, x_2,\ldots ,x_n$ of the points of the space $X$ satisfying the majorization property. Then $X\hookrightarrow \mathbb R^1$.
\end{con}

The following statement sheds some more light on the case $n=4$.

\begin{thm}\label{t10}
Let $(X,\delta)$ be an ordinal space with $|X|=4$. The following conditions are equivalent:
\begin{itemize}
  \item [(i)] $X\hookrightarrow \mathbb R^1$.
  \item [(ii)] There exists an enumeration $x_1, x_2, x_3, x_4$ of the points of the space $X$ satisfying the majorization property.
  \item[(iii)] There exists an enumeration $x_1, x_2, x_3, x_4$ of the points of the space $X$ such that the following two conditions hold:
      \begin{enumerate}
        \item $\delta(x_1, x_2)<\delta(x_1, x_3)<\delta(x_1, x_4)>\delta(x_2, x_4)>\delta(x_3, x_4)$\\ and
        $\delta(x_2, x_3)<\delta(x_1, x_3),\delta(x_2, x_4)$,
        \item $(\delta(x_1, x_3)<\delta(x_2, x_4))\Leftrightarrow(\delta(x_1, x_2)<\delta(x_3, x_4))$\\ and
            $(\delta(x_1, x_3)=\delta(x_2, x_4))\Leftrightarrow(\delta(x_1, x_2)=\delta(x_3, x_4))$.
            \end{enumerate}
\end{itemize}
\end{thm}
\begin{proof}
The implications (i)$\Rightarrow$(ii)  and (ii)$\Rightarrow$(iii) follow immediately from Propositions~\ref{p101} and~\ref{p102}.

(iii)$\Rightarrow$(i): Let us assume the validity of (iii). Set $\delta_{ij}:=\delta(x_i, x_j)$ for $i,j\in\{1,2,3,4\}$ and for positive real numbers $a,b,c$ define a map $\Phi\colon X\to \RR^1$ by the following rule:
\begin{equation}\label{e123}
\Phi(x_1)=0, \ \Phi(x_2)=a, \ \Phi(x_3)=a+b, \ \Phi(x_4)=a+b+c.
\end{equation}
We show that $\Phi$ is an embedding if $a,b$ and $c$ are chosen in an appropriate way. For this we distinguish the several cases of relationships not being fixed by conditions (1) and (2).

If $\delta_{13}=\delta_{24}$, by (2) it follows $\delta_{12}=\delta_{34}$ and we have the following possibilities:
\begin{eqnarray}
\delta_{14}>\delta_{24}=\delta_{13}>\delta_{23}>\delta_{12}=\delta_{34},\label{d1}\\
\delta_{14}>\delta_{24}=\delta_{13}>\delta_{23}=\delta_{12}=\delta_{34},\label{d2}\\\
\delta_{14}>\delta_{24}=\delta_{13}>\delta_{12}=\delta_{34}>\delta_{23}.\label{d3}
\end{eqnarray}
The mapping $\Phi$ is an embedding if $b>a=c$, $a=b=c$, or $a=c>b$, respectively, for \eqref{d1}, \eqref{d2}, or \eqref{d3}, respectively.
Since all the relations between $a,b$ and $c$ are possible, in each of the three cases one gets embeddability.

We now consider the case $\delta_{13}>\delta_{24}$ and obtain the following possibilities:
\begin{eqnarray}
\delta_{14}>\delta_{13}>\delta_{12}>\delta_{24}>\delta_{23}>\delta_{34},\label{d4}\\
\delta_{14}>\delta_{13}>\delta_{12}>\delta_{24}>\delta_{23}=\delta_{34},\label{d5}\\
\delta_{14}>\delta_{13}>\delta_{12}>\delta_{24}>\delta_{34}>\delta_{23},\label{d6}\\
\delta_{14}>\delta_{13}>\delta_{12}=\delta_{24}>\delta_{23}>\delta_{34},\label{d7}\\
\delta_{14}>\delta_{13}>\delta_{12}=\delta_{24}>\delta_{23}=\delta_{34},\label{d8}\\
\delta_{14}>\delta_{13}>\delta_{12}=\delta_{24}>\delta_{34}>\delta_{23},\label{d9}\\
\delta_{14}>\delta_{13}>\delta_{24}>\delta_{12}>\delta_{23}>\delta_{34},\label{d10}\\
\delta_{14}>\delta_{13}>\delta_{24}>\delta_{12}>\delta_{23}=\delta_{34},\label{d11}\\
\delta_{14}>\delta_{13}>\delta_{24}>\delta_{12}>\delta_{34}>\delta_{23},\label{d12}\\
\delta_{14}>\delta_{13}>\delta_{24}>\delta_{12}=\delta_{23}>\delta_{34},\label{d13}\\
\delta_{14}>\delta_{13}>\delta_{24}>\delta_{23}>\delta_{12}>\delta_{34}.\label{d15}
\end{eqnarray}
The mapping $\Phi$ is an embedding for
\eqref{d4} iff $a>b+c$ and $b>c$,
\eqref{d5} iff $a>b+c$ and $b=c$,
\eqref{d6} iff $a>b+c$ and $b<c$,
\eqref{d7} iff $a=b+c$ and $b>c$,
\eqref{d8} iff $a=b+c$ and $b=c$,
\eqref{d9} iff $a=b+c$ and $b<c$,
\eqref{d10} iff $a<b+c$ and $a>b>c$,
\eqref{d11} iff $a<b+c$ and $a>b=c$,
\eqref{d12} iff $a<b+c$ and $b<c<a$,
\eqref{d13} iff $a=b>c$,
and \eqref{d15} iff $b>a>c$.

Again all the relations between $a,b$ and $c$ are realizable, showing embeddability in each case considered.
The case $\delta_{13}<\delta_{24}$ can by considered analogously to the case $\delta_{13}>\delta_{24}$.
\end{proof}

The proof of Theorem \ref{t10} in fact contains a listing of all nonisomorphic ordinal spaces of cardinality $4$. It particularly provides that there exist
$14$ such spaces. Here to each case with $\delta_{13}>\delta_{24}$ there is a corresponding isomorphic one with $\delta_{13}<\delta_{24}$.\vspace{3mm}

\textbf{Open problem.} \emph{How many non-isomorphic ordinal spaces $X$ with $|X|=n$ does there exist for fixed $n\in \mathbb N^+$?}\vspace{3mm}

Note that an enumeration of the points of an ordinal space is majorizing iff the inverse enumeration is, and that by Proposition \ref{p102} (i)
there are no further majorizing enumerations. The majorization property is relatively complicated since it not only compares ``intervals'' between successive points under the enumeration. The following example shows that the property cannot be expressed on the level of ``intervals'' with successive endpoints.

\begin{exa}
Let $x_1,x_2,\ldots ,x_7$ be points in ${\mathbb R}^1$ with $x_1< x_2<\ldots < x_6<x_7$ and $d(x_1,x_2)=2.00001$, $d(x_2,x_3)=2.000001$,
$d(x_3,x_4)=4.001$, $d(x_4,x_5)=1.1$, $d(x_5,x_6)=3.01$, and $d(x_6,x_7)=4.0001$, where $d$ denotes the Euclidean distance.

One easily sees that all distances between distinct points in $X=\{x_1,x_2,\ldots ,x_7\}$ are different, and with $z_{\{i,j\}}$ being the nearest integer to $d(x_i,x_j)$  for different $i,j\in\{1,2,\ldots ,7\}$ it holds $|d(x_i,x_j)-z_{\{i,j\}}|<0.2$. Moreover, $z_{\{i,j\}}=8$ iff $\{i,j\}\in\{\{1,4\},\{4,7\},\{3,6\}\}$, and $d(x_1,x_4)=8.001011$,  $d(x_4,x_7)=8.1101$, and $d(x_3,x_6)=8.111$.

Consider the semimetric $\delta$ on $X$ defined by
$$
\delta (x,x')=
\begin{cases}
d(x_4,x_7), &\text{if} \ \, \{x,x'\}=\{x_1,x_4\};\\
d(x_1,x_4), &\text{if} \ \, \{x,x'\}=\{x_4,x_7\};\\
d(x,x'), &\text{otherwise.}
\end{cases}
$$

Assuming that there exists an enumeration of the points in $X$ satisfying the majorization property with respect to $\delta$, then by Proposition \ref{p102} (i) the already given enumeration is majorizing. For this enumeration it holds
$(1,3,4)\prec (4,6,7)$ since $\delta(x_1,x_3)<\delta(x_6,x_7)$ and $\delta(x_3,x_4)<\delta(x_4,x_6)$, but
$\delta(x_1,x_4)=8.1101 > 8.001011=\delta(x_4,x_7)$, contradicting majorizing.

In order to show that the enumeration satisfies the majorization property in the weaker sense that only ``intervals'' with successive endpoints are considered, it is enough only to consider the sequences $(1,2,3,4)$ and $(4,5,6,7)$. The reason for this is that the order relation of $\delta(x_i,x_j)$ and $\delta(x_k,x_l)$ can be only different from that of $d(x_i,x_j)$ and $d(x_k,x_l)$ if
$\{\{i,j\},\{k,l\}\}=\{\{1,4\},\{4,7\}\}$. We are done by showing that neither $(1,2,3,4)\prec (4,5,6,7)$ nor $(4,5,6,7)\prec (1,2,3,4)$.

The case $(1,2,3,4)\prec (4,5,6,7)$ is impossible since $\delta(x_4,x_5), \delta(x_5,x_6),\delta(x_6,x_7)<\delta(x_3,x_4)$. Assuming $(4,5,6,7)\prec (1,2,3,4)$, it follows $\delta(x_4,x_5)<\delta(x_1,x_2)$, $\delta(x_5,x_6)<\delta(x_2,x_3)$ or $\delta(x_5,x_6)<\delta(x_1,x_2)$, $\delta(x_4,x_5)<\delta(x_2,x_3)$, which is obviously false.
\end{exa}

We finish discussing embeddability of finite ordinal spaces $(X,\delta)$ in ${\mathbb R}^1$ by looking at the system of equivalence classes  generated by the relation `$=$' on the set of unordered pairs $\{x,y\}$, $x\neq y$, $x,y \in X$ (see Remark~\ref{r1.3}), its cardinality, and the cardinality of the equivalence classes themselves. The set of these equivalence classes
is denoted below by $\Delta(X)$. For the equivalence classes $\delta_i$ and $\delta_j$ we write $\delta_i>\delta_j$ if $\delta(x,y)>\delta(z,w)$ for some $\{x,y\}\in \delta_i$ and some $\{z,w\}\in \delta_j$.

\begin{prop}\label{p6.8}
Let $(X,\delta)$ be a finite ordinal space, $|X|\geqslant 2$, and $X\hookrightarrow \RR^1$. Then the following conditions are equivalent:
\begin{enumerate}
  \item [(i)] $|\Delta(X)|=|X|-1$.
  \item [(ii)] If $\Delta(X)=\{\delta_1,\delta_2,\ldots ,\delta_k\}$ and
      $$
      \delta_1>\delta_2>\ldots >\delta_k,
      $$
      then $|\delta_k|=|X|-1$.
  \item [(iii)]  If $\Delta(X)=\{\delta_1,\delta_2,\ldots ,\delta_k\}$ and
      $$
      \delta_1>\delta_2>\ldots >\delta_k,
      $$
      then $|\delta_i|=i$ for every $i=1,\ldots ,k$.
\end{enumerate}
\end{prop}
\begin{proof}
We can assume that $X=\{x_1,x_2,\ldots ,x_n\}\subset {\mathbb R^1}$ for a natural number $n$, that
$x_1<x_2<\ldots <x_n$, and that $\delta$ is defined by the Euclidean distance.

If $\Delta(X)=\{\delta_1,\delta_2,\ldots ,\delta_k\}$ with
$\delta_1>\delta_2>\ldots >\delta_k$ and $|\delta_i|=i$ for every $i=1,2,\ldots ,k$, then we have $k(k+1)/{2}=n(n-1)/2$ pairs of different points from $X$ implying $k=n-1$. This shows the implication
(iii)$\Rightarrow$(ii). If (ii) is valid, it obviously holds that $\delta_k=\{\{x_1,x_2\},\{x_2,x_3\},\ldots ,\{x_{n-1},x_n\}\}$, which shows condition (i).

(i)$\Rightarrow$(iii). Assume that condition (i) holds, and let $\Delta(X)=\{\delta_1,\delta_2,\ldots ,\delta_{n-1}\}$ with $\delta_1>\delta_2>\ldots >\delta_{n-1}$. Then the pairs
$$
\{x_1,x_2\},\dots,\{x_1,x_{n-1}\},\{x_1,x_n\}
$$
belong to different ones of the given equivalence classes, more precisely,
$$
\{x_1,x_n\} \in \delta_1,\, \ \{x_1,x_{n-1}\}\in \delta_2,\, \ \dots,\{x_1,x_{2}\}\in \delta_{n-1}.
$$
Since $\{x_1,x_n\}$ is a single diametrical pair in $X$, it is clear that
$$
\delta_1=\{\{x_1,x_{n}\}\}.
$$
Arguing as above we see that
$$
\{x_2,x_{n}\} \in \delta_2,\, \ \{x_2,x_{n-1}\}\in \delta_3,\, \ \dots,\{x_2,x_{3}\}\in \delta_{n-1}.
$$
Since for all pairs $\{x_i,x_j\}$ which are not considered above the equality $\delta(x_i,x_j)<\delta(x_2,x_{n})$ holds, we obtain that
$$
\delta_2=\{\{x_1,x_{n-1}\},\{x_2,x_{n}\}\}.
$$
Repeating this procedure with the pairs
$$
\{x_i,x_{n}\},\dots,\{x_i,x_{n-1}\},\{x_i,x_{i+1}\}
$$
for every $i=3,\ldots ,n-1$, we establish condition (iii).
\end{proof}

\begin{prop}\label{p210}
Let $(X,\delta)$ be a finite ordinal space, $|X|\geqslant 2$, and $X \hookrightarrow \RR^1$. Then the following conditions hold:
\begin{enumerate}
  \item [(i)] $|\Delta(X)|\geqslant|X|-1$.
  \item [(ii)] If $\Delta(X)=\{\delta_1,\delta_2,\ldots ,\delta_k\}$ and
      $$
      \delta_1>\delta_2>\ldots >\delta_k,
      $$
      then
      \begin{equation}\label{e211}
        |\delta_1|=1, \, \ |\delta_i|\leqslant i,\ \, i=2,\ldots ,k.
      \end{equation}
\end{enumerate}
\end{prop}
\begin{proof}
We shall prove conditions (i) and (ii) by induction on $|X|$.
If $|X|=2$, then (i) and (ii) are true. Now let $n\geqslant 2$ be such that condition (i) holds for every $X$ with $|X|\leqslant n$ and let $|X|=n+1$. We have to prove that
\begin{equation}\label{e210}
  |\Delta(X)|\geqslant n.
\end{equation}

Let $\{x_0,y_0\}$ be the diametrical pair of $X$, i.e., $\{x_0,y_0\}\in \delta_1(X)$. Consider the space $(Y,\delta)$ with $Y=X\setminus\{x_0\}$. Since $X \hookrightarrow \RR^1$, we have $Y \hookrightarrow \RR^1$ and according to assumption $|\Delta(Y)|\geqslant|Y|-1=n-1$. It is clear that $\Delta(Y)\subseteq \Delta(X)$ and $\delta_1(X)\notin \Delta(Y)$. Hence inequality~(\ref{e210}) follows.

Let $n\geqslant 2$ be such that condition (ii) holds for every $X$ with $|X|\leqslant n$ and let $|X|=n+1$.
We can assume that $X=\{x_0,x_1,\ldots,x_n\}\subset \RR^1$ with $x_0<x_1<\ldots<x_n$. Then $\{x_0,x_n\}\in \delta_1(X)$. Consider the space $(Y,\delta)$ with $Y=X\setminus\{x_0\}$. Since $X \hookrightarrow \RR^1$, we have $Y \hookrightarrow \RR^1$ and according to assumption
\begin{equation}\label{e212}
        |\delta_1(Y)|=1, \, \ |\delta_i(Y)|\leqslant i,\ \,  i=2,\ldots ,k,
\end{equation}
where $k=|\Delta(Y)|$.

It is clear that for every $\delta_Y\in \Delta(Y)$ there exists a $\delta_X \in  \Delta(X)$ such that $\delta_Y\subseteq \delta_X$.
Consider the pairs of points
$$
\{x_0,x_1\}, \{x_0,x_2\}, \ldots , \{x_0,x_n\}.
$$
Note that all these pairs belong to different classes of $\Delta(X)$. Let $N_0$ be the set of all the indices $i\in \{1,\ldots ,n\}$ for which $\delta(x_0,x_i)\neq\delta(x_p,x_q)$ whenever $1\leqslant p<q\leqslant n$. It is clear that $n\in N_0$,
$\delta_1(X)\notin \Delta(Y)$ and
\begin{equation}\label{e219}
  \delta_1(X)=\{\{x_0,x_n\}\}.
\end{equation}
Let
 $$
    \delta_1(X)>\delta_2(X)>\ldots >\delta_l(X)
 $$
be the equivalence classes of the space $X$, $l>k$.
Observe that there exists only three possibilities for the class $\delta_i(X)$, namely
$$
\delta_i(X)=
\begin{cases}
\{\{x_0,x_j\}\}, &\text{for some  }\  j\in N_0 \ \text{ or }\\
\{\{x_0,x_{j}\}\}\cup \delta_p(Y), &\text{for some  }\  j\in N \! \setminus \! N_0 \ \text{and some  }\ p<i  \ \text{ or }\\
\delta_p(Y), &\text{for some }\ p<i.
\end{cases}
$$
Hence, taking into considerations relations~(\ref{e212}) we
have
$$
|\delta_i(X)|\leqslant 1+|\delta_p(Y)|\leqslant 1+p\leqslant i.
$$
This inequality and~(\ref{e219}) establish relations~(\ref{e211}).
\end{proof}

\section{Embedding of ordinal spaces in higher dimensional Euclidean spaces}\label{hdes}

\noindent In 1928 K.~Menger proved in \cite{Me28} that in order to verify whether an abstract semimetric space $X$ is embeddable in $\RR^n$, one only needs to verify the embeddability of each of its $n+3$ point subsets. This result is a consequence of a more general theorem which is beyond the scope of the current paper.

Naturally the following question arises: Does the analog of Menger's consequence hold for ordinal spaces?

\begin{con}
Let $(X,\delta)$ be a finite ordinal space and let \mbox{$n\in \NN^+$}. Then the following conditions are equivalent:
\begin{itemize}
  \item [(i)] $X \hookrightarrow \mathbb R^n$,
  \item [(ii)] For every $A\subseteq X$ with $|A|\leqslant n+3$ it holds $A \hookrightarrow \mathbb R^n$.
\end{itemize}
\end{con}

Note that the implication (i) $\Rightarrow$ (ii) is trivial but the converse implication is not trivial, even in the case $n=1$.

Denote by $\DP(X)$ the set of all diametrical pairs of points of the ordinal space $X$, i.e., pairs of points $x,y$ with maximal possible $\delta(x,y)$. Let $S$ be a set of points on the plane.
Consider line segments in this plane with endpoints $x$ and $y$, $x,y \in S$, such that the distance between $x$ and $y$ is equal to some fixed positive real number $d$. Denote by $S^d$ the geometric object which is the union of all such line segments. We shall say that $S^d$ contains a closed polygonal chain if such a chain can be obtained from $S^d$ by deleting some number of line segments.

Recall that a Reuleaux polygon is a plane convex set of constant width $d$ whose boundary consists of a finite (necessarily odd) number of circular arcs of radius $d$, where the center of each circle arc is located at the end of one of the arcs. We shall say that an odd closed polygonal chain consisting of segments of length $d$ and joining vertices of some Reuleaux polygon of width $d$ is a frame of this Reuleaux polygon, see Figure~\ref{f4}.

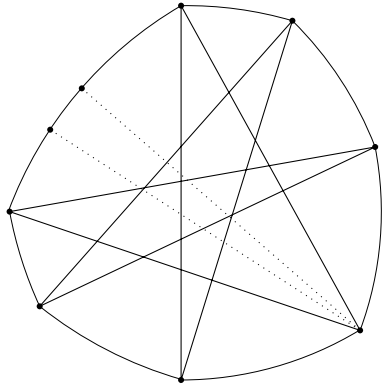
\begin{figure}[ht]
\begin{center}
\begin{tikzpicture}[scale=0.2]
\draw (2.2,5.5) -- (19,24.5) -- (11.6,0.6) -- (11.6,25.5) -- (23.5,3.9)  -- (0.2,11.8) -- (24.5,16.1) -- cycle;
\draw[dotted] (23.5,3.9) -- (5,20);
\draw[dotted] (23.5,3.9) -- (2.9,17.25);

\tkzDefPoint(11.6,25.5){A}
\tkzDefPoint(5,20){B}
\tkzDefPoint(0.2,11.8){C}
\tkzDefPoint(25,2.6){O}
\tkzDrawArc(O,A)(C)

\tkzDefPoint(19,24.5){A}
\tkzDefPoint(11.6,25.5){C}
\tkzDefPoint(12,0.6){O}
\tkzDrawArc(O,A)(C)

\tkzDefPoint(24.5,16.1){A}
\tkzDefPoint(21.8,20.8){B}
\tkzDefPoint(19,24.5){C}
\tkzDefPoint(2.2,7.5){O}
\tkzDrawArc(O,A)(C)

\tkzDefPoint(23.5,3.9){A}
\tkzDefPoint(24.9,11){B}
\tkzDefPoint(24.5,16.1){C}
\tkzCircumCenter(A,B,C)\tkzGetPoint{O}
\tkzDrawArc(O,A)(C)

\tkzDefPoint(11.6,0.6){A}
\tkzDefPoint(16.8,1.2){B}
\tkzDefPoint(23.5,3.9){C}
\tkzDefPoint(11.6,23.5){O}
\tkzDrawArc(O,A)(C)

\tkzDefPoint(11.6,0.6){C}
\tkzDefPoint(2.2,5.5){A}
\tkzDefPoint(19,26.5){O}
\tkzDrawArc(O,A)(C)

\tkzDefPoint(0.2,11.8){A}
\tkzDefPoint(1,8.6){B}
\tkzDefPoint(2.2,5.5){C}
\tkzCircumCenter(A,B,C)\tkzGetPoint{O}
\tkzDrawArc(O,A)(C)

\foreach \i in {(2.2,5.5), (19,24.5), (11.6,0.6), (11.6,25.5), (23.5,3.9), (0.2,11.8), (24.5,16.1), (5,20), (2.9,17.25)}
  \fill[black] \i circle (0.2cm);
\end{tikzpicture}
\end{center}
\caption{The frame of Reuleaux polygon is not necessarily a regular star polygon.}\label{f4}
\end{figure}

In~\cite[p.~230]{BMP05} it was stated that the maximum number of times that the largest distance occurs among $n\neq 2$ points in the plane is $n$. All such extremal configurations were described without a reference for the proof. In the next theorem formulated in terms of ordinal spaces we present an idea for an independent proof of the given assertion. It is based on the unproved  Conjectures~\ref{c3} and~\ref{c4}, which are of some particular interest beyond the aim of this paper.

\begin{thm}\label{t22}
Let $(X,\delta)$ be a finite ordinal space. If $X\hookrightarrow \RR^2$,
then $|\DP(X)|\leqslant |X|$. Moreover, if $\Phi\colon X\hookrightarrow \RR^2$ is an embedding and $|\DP(X)|=|X|$, then $S^{d}$ is the union of $2k+1$-star polygon (diameters of some Reuleaux polygon) with $n-2k-1$ line segments, where $S=\Phi(X)$, $d=\diam \Phi(X)$, and $n=|X|$.
\end{thm}
\begin{proof}[\textbf{Sketch of the proof}] Let $X\hookrightarrow \RR^2$. Suppose first that $S^{d}$ contains some closed polygonal chain $Ch$.
\begin{con}\label{c3}
If the diameter of a closed equilateral polygonal chain is equal to the length of a line segment of this chain, then this chain is a  frame of some Reuleaux polygon.
\end{con}
\begin{cor}
If the diameter of a closed equilateral polygonal chain is equal to the length of a line segment of this chain, then the number of line segments of this chain is odd.
\end{cor}
Since $d=\diam \Phi(X)$, we have that $d$ is the diameter of $S^d$ and according to Conjecture~\ref{c3} the chain $Ch$ is a frame of some Reuleaux polygon.
\begin{con}\label{c4}
Let $F_1$ and $F_2$  be frames of some Reuleaux polygons with $\diam (F_1)=\diam (F_2)=d$ lying nonidentically in Euclidean plane. Then $\diam (F_1 \cup F_2)>d$.
\end{con}
According to Conjecture~\ref{c4} the chain $Ch$ is a single closed polygonal chain in $S^d$. Suppose that the number of its line segments is $2k+1$, where $k$ is an integer such that $2k+1\leqslant n$. It is easy to show that for every line segment which  belongs to $S^d$ and does not belong to $Ch$ the following property holds: one of endpoints of this segment coincides with some vertex of $Ch$ and another one lies at the opposite circular arc of Reuleaux polygon having the frame $Ch$, see dotted lines at Figure~\ref{f4}. In this case the inequality $|\DP(X)|\leqslant|X|$ follows easily and the equality is attained if and only if the number of the above mentioned line segments is equal to $n-2k-1$.

Suppose now that $S^{d}$ does not contain any closed polygonal chain. In this case we can consider $S^{d}$ as a graph $G=(V,E)$ for which the vertex set $V$ coincides with the vertex set of  $S^{d}$ and the set of edges $E$ coincides with the set of line segments of $S^{d}$. It is clear that $G$ is a forest and the inequality $|E|<|V|$ holds. Hence the inequality $|\DP(X)|<|X|$ follows.
\end{proof}

The problem of finding the minimal number of different distances determined by a set of $n$ points in the plane was formulated by P. Erd\H{o}s in~\cite{Er46},  where he gave the first estimation for this value. His result was improved several times and at the present day the best estimation is found in~\cite{Ch92}. Estimations of the numbers of second largest, second smallest and smallest distances for $n$ points in the plane are found in \cite{Ve87}, \cite{Br92} and \cite{Ha75}, respectively. These results and Theorem~\ref{t22} give us necessary conditions for the embedding of ordinal spaces in $\RR^2$.

\begin{prop}\label{p24}
Let $(X,\delta)$ be a finite ordinal space, $|X|=n\geqslant 2$, $X \hookrightarrow \RR^2$, and let $\Delta(X)=
\{\delta_1,\delta_2,\ldots,\delta_k\}$ with
      $$
      \delta_1>\delta_2>\ldots >\delta_k.
      $$
Then the following conditions hold:
\begin{enumerate}
  \item [(i)] $|\Delta(X)|\geqslant n^{4/5}/(\log n)^c$ for  some $c>0$ for all $n\geqslant 10$.
  \item [(ii)] $|\delta_1|\leqslant n$.
  \item [(iii)] $|\delta_2|\leqslant \lfloor \frac32n \rfloor$.
  \item [(iv)] $|\delta_{k-1}|<  \frac{24}{7}n $.
  \item [(v)] $|\delta_k|\leqslant \lfloor 3n-\sqrt{12n-3} \rfloor$.
\end{enumerate}
\end{prop}
The asymptotic behavior of the number of distinct distances determined by a set of $n$ points in three and higher dimensional spaces were found in~\cite{Ar04}.
It was proved in~\cite{Gr56,He56,St57} that in a three-dimensional space the maximum number of largest distances among $n$ points is $2n-2$. The corresponding results provide necessary conditions for embeddability of ordinal spaces in three and higher dimensional Euclidean spaces.

A semimetric space $X$ with $|X|=n+1$ is said to be \emph{irreducibly embeddable} in $\RR^n$ if it is isometric to an independent $(n+1)$-subset of $\RR^n$.
Recall that the Cayley--Menger determinant is the polynomial
$$
D_{k}(x_0,x_1,\ldots ,x_k)=
\begin{vmatrix}
    0&1&1&\ldots &1\\
    1&0&d^{2}(x_0,x_1)&\ldots &d^{2}(x_0,x_k)\\
    1&d^{2}(x_1,x_0)&0&\ldots &d^{2}(x_1,x_k)\\
    \vdots&\vdots&\vdots&\ddots&\vdots\\
    1&d^{2}(x_k,x_0)&d^{2}(x_k,x_1)&\ldots &0\\
\end{vmatrix},
$$ where $x_0,x_1,\ldots ,x_k$ are some points of a semimetric space $(X,d)$.

To prove Proposition~\ref{p10} we shall use the following theorem of  L. Blumenthal, see \cite[p.100]{Bl}.
\begin{thm}
A necessary and sufficient condition that a semimetric space $X$ with $|X|=n+1$ and $X=\{x_0,x_1,\ldots ,x_n\}$ is irreducibly embeddable in $\RR^n$ is that
\begin{equation}\label{sgn}
\sgn D_k(x_0,x_1,\ldots ,x_k)=(-1)^{k+1}, \quad k=1,2,\ldots ,n.
\end{equation}
\end{thm}

\begin{prop}\label{p10}
For any ordinal space $X$ with $|X|=n+1$ there is an embedding $\Phi\colon X \to \mathbb R^n$, such that $\Phi(X)$ is an independent subset of $\RR^n$.
\end{prop}
\begin{proof}
To prove this proposition apply Proposition~\ref{p1} to a finite $(X,\delta)$ where $X=\{x_0,x_1,\ldots ,x_n\}$. Let $(X,d)$ be a realization of  $(X,\delta)$ depending on parameters $a$ and $\ve$ as in the proof of Proposition~\ref{p1}. It is easy to see that $D_k(x_0,x_1,\ldots ,x_k)$ is a polynomial $P_k(d(\cdot,\cdot),\ldots ,d(\cdot,\cdot))$ depending on $d(x_i,x_j)$, $0\leqslant i,j\leqslant k$, $i\neq j$, and that $d(x_i,x_j)\to a$ as $\ve\to 0$ for $i\neq j$. Moreover, it is clear that $$
\lim\limits_{d(x_i,x_j)\to a}P_k(d(\cdot,\cdot),\ldots ,d(\cdot,\cdot))= P_k(a,\ldots ,a).
$$
Hence in order to prove~(\ref{sgn}) it suffices to establish the equality $\sgn P_k(a,\ldots ,a)=(-1)^{k+1}$, $k=1,2,\ldots ,n$. Now
$$
P_k(a,\ldots ,a)=
\left|
  \begin{array}{cccccc}
    0 & 1 & 1 &  \ldots & 1 \\
    1 & 0 & a^2&  \ldots & a^2\\
    1 & a^2& 0 &  \ldots & a^2\\
    \vdots &  \vdots & \vdots & \ddots & \vdots\\
    1 & a^2& a^2&  \ldots & 0 \\
  \end{array}
\right| =
\left|
  \begin{array}{cccccc}
    0 & 1 & 1 &  \ldots & 1 \\
    1 & -a^2& 0 &  \ldots & 0 \\
    1 & 0 & -a^2&  \ldots & 0 \\
    \vdots &  \vdots & \vdots & \ddots & \vdots \\
    1 & 0 & 0 &  \ldots & -a^2\\
  \end{array}
\right|
$$
$$
=\left|
  \begin{array}{cccccc}
    0 & 1 & 1 &  \ldots & 1 \\
    0 & -a^2& 0 &  \ldots & a^2\\
    0 & 0 & -a^2&  \ldots & a^2\\
    \vdots &  \vdots & \vdots & \ddots & \vdots \\
    1 & 0 & 0 &  \ldots & -a^2\\
  \end{array}
\right|=
\left|
  \begin{array}{cccccc}
    0 & 1 & 1 &  \ldots & k+1 \\
    0 & -a^2& 0 &  \ldots & 0 \\
    0 & 0 & -a^2&  \ldots & 0 \\
    \vdots &  \vdots & \vdots & \ddots & \vdots \\
    1 & 0 & 0 & \ldots & -a^2\\
  \end{array}
\right|
 = (-1)^{k+1}(k+1)a^{2k}.
$$
In other words, the realization $(X,d)$ of an ordinal space $(X,\delta)$ is irreducibly embeddable in $\RR^n$ for sufficiently small $\ve$.
\end{proof}

\section{Distance between ordinal spaces}\label{dist}

\noindent Usually, the concept of a metric is associated with the distance between points of a certain space. But in mathematics there are a lot of metrics defined not on points but on completely different mathematical objects. A large number of distances is collected in~\cite{DD16}.
Among these distances one can distinguish distances on graphs, matrices, strings and permutations, etc. Below we propose a distance defined on classes of equivalences of ordinal spaces with a fixed number of points. The most similar metric to this one is the well known Gromov--Hausdorff metric, which is also a metric on classes of spaces. The definition of the Gromov--Hausdorff distance can be found, e.\,g., in~\cite[p. 254]{BBI}.

Let $(X,\delta_X)$  and $(Y,\delta_Y)$ be ordinal spaces with $|X|=|Y|<\infty$. Define the distance between $X$ and $Y$ as
\begin{equation}\label{eq2}
d_{ord}(X,Y)=\min\limits_{f}\frac{1}{8}\Bigl|\{{(x,y,z,w)\, |\, \delta_X(x,y,z,w)}\neq\delta_Y(f(x),f(y),f(z),f(w))\}\Bigr|,
\end{equation}
where $f\colon X\to Y$ is a bijection and $(x,y,z,w)$ is an ordered set of points \mbox{$x,y,z,w \in X$.}

\begin{rem}
The geometric sense of the distance $d_{ord}$ is the minimal number of relations between two pairs of points which are not preserved under any bijection. We need the factor $\frac{1}{8}$
since
\begin{equation}
\begin{split}
\delta_X(x,y,z,w)=\delta_X(x,y,w,z)=\delta_X(y,x,z,w)=\delta_X(y,x,w,z)\\
=-\delta_X(z,w,x,y)=-\delta_X(w,z,x,y)=-\delta_X(z,w,y,x)=-\delta_X(w,z,y,x).
\end{split}
\end{equation}
\end{rem}

For an ordinal space $X$ we denote by $\operatorname{is}(X)$ the \emph{isomorphic type} of this space, i.e., $\operatorname{is}(X)$ is the class of all ordinal spaces that are isomorphic to $X$. Under the distance $d_{ord}(\is(X),\is(Y))$ between isomorphic types $\is(X)$ and $\is(Y)$  we understand the distance $d_{ord}(X,Y)$ for any $X\in \is(X)$ and any $Y\in \is(Y)$.

\begin{thm}\label{p3}
$d_{ord}$ is a metric on the set of isomorphic types of finite ordinal spaces with a fixed number of points.
\end{thm}
\begin{proof}
The symmetry of $d_{ord}$ and the equivalence $((d_{ord}(\is(X),\is(Y))=0)\Leftrightarrow (\operatorname{is}(X)=\operatorname{is}(Y)))$ are almost evident.

Let us prove the triangle inequality. Let $d_{ord}(X,Y)$ and $d_{ord}(Y,Z)$ be distances with minima attained in~(\ref{eq2}) at bijections $f_{1}\colon X\to Y$ and $f_{2}\colon Y\to Z$ respectively. Define
\begin{equation*}
\Omega_X^1=\{{(x,y,z,w)\, |\, \delta_X(x,y,z,w)}\neq\delta_Y(f_1(x),f_1(y),f_1(z),f_1(w))\},
\end{equation*}
where $(x,y,z,w)$ is an ordered set of points $x,y,z,w \in X$, and define
\begin{equation*}\label{eq3}
\Omega_Y^1=\{{(x,y,z,w)\, |\, \delta_Y(x,y,z,w)}\neq\delta_Z(f_2(x),f_2(y),f_2(z),f_2(w))\},
\end{equation*}
where $(x,y,z,w)$ is an ordered set of points \mbox{$x,y,z,w \in Y$.}
It is clear that
\begin{equation}\label{eq44}
d_{ord}(X,Y)=\frac{1}{8}|\Omega_X^1|\, \ \text{ and } \, \ d_{ord}(Y,Z)=\frac{1}{8}|\Omega_Y^1|.
\end{equation}
Define
\begin{equation}\label{eq5}
\Omega=\Omega_X^1\cup f^{-1}_1(\Omega_Y^1),
\end{equation}
where
\begin{eqnarray*}
&&f^{-1}_1(\Omega_Y^1)=\\
&&\left\{\big(f^{-1}(x),f^{-1}(y),f^{-1}(z),f^{-1}(w)\big) \, | \, \delta_Y(x,y,z,w)\neq \delta_Z\big(f_2(x),f_2(y),f_2(z),f_2(w)\big)\right\}.
\end{eqnarray*}

Consider the composition $f=f_2\circ f_1\colon X \to Z$ and put
\begin{equation*}
d_f(X,Z)=\frac{1}{8}\Bigl|\{{(x,y,z,w)\, |\, \delta_X(x,y,z,w)}\neq\delta_Z(f(x),f(y),f(z),f(w))\}\Bigr|,
\end{equation*}
where $x,y,z,w \in X$.
The geometric sense of the set $\Omega$ is the following: $\Omega$ is the set of ordered groups of four points from $X$ for which the mapping $f$ in general does not preserve the relations.
Using the definition of $d_{ord}(X,Z)$,~(\ref{eq5}) and~(\ref{eq44}), we obtain
\begin{multline}
 d_{ord}(X,Z)\leqslant d_f(X,Z)\leqslant \frac{1}{8}|\Omega|\leqslant \frac{1}{8}|\Omega_X^1|+ \frac{1}{8} |f^{-1}(\Omega_Y^1)|\\
 =\frac{1}{8}|\Omega_X^1|+ \frac{1}{8} |\Omega_Y^1|=d_{ord}(X,Y)+d_{ord}(Y,Z).
\end{multline}
\end{proof}

\begin{rem}
The formulation of Theorem~\ref{p3} holds also for ordinal structure without the axioms of transitivity (iv), (v) and (vi).
\end{rem}

\section*{Acknowledgements}
The second author was supported by H2020-MSCA-RISE-2014, Project number 645672 (AMMODIT: ``Approximation Methods for Molecular Modelling and Diagnosis Tools''). The research of the second author was also partially supported by Project 0117U006353 from the Department of Targeted Training of Taras Shevchenko National University of Kyiv at the NAS of Ukraine and by the National Academy of Sciences of Ukraine within scientific research works for young scientists, Project 0117U006050.

\end{document}